\documentclass[10pt,a4paper,draft]{article}
\usepackage{bbm}
\usepackage{pifont}
\usepackage{bbding}
\usepackage{amsmath}
\usepackage{mathrsfs}
\usepackage{amsfonts}
\usepackage{amssymb}
\usepackage{amsfonts,amssymb,amsmath,indentfirst,amsthm}
\usepackage{epsfig}

\numberwithin{equation}{section}
\newenvironment{abs}{\textbf{Abstract}\mbox{  }}{ }
\newenvironment{key words}{\emph{\texttt{Keywords}}\mbox{  }}{ }
\newtheorem{theorem}{Theorem}[section]
\newtheorem{lemma}[theorem]{Lemma}

\newtheorem{proposition}[theorem]{Proposition}

\newtheorem{definition}[theorem]{Definition}
\renewenvironment{proof}{\noindent{\textbf{Proof.}}}{\hfill$\Box$}
\theoremstyle{remark}

\theoremstyle{plain}

\makeatletter
    
    \newcommand{\Rmnum}[1]{\expandafter\@slowromancap\romannumeral #1@}
\makeatother

\begin{document}

\title{\textbf{Two-bubble nodal solutions for slightly subcritical Fractional Laplacian}
\thanks{The second author was supported by the National Natural Science Foundation of China (Grant Nos.11271299, 11001221) and the Fundamental Research Funds for the Central Universities (Grant No. 3102015ZY069).}}
\author{Qianqiao Guo, Yunyun Hu}
\date{}
\maketitle

\noindent
\begin{abs}
 In this paper, we consider the existence of nodal solutions with two bubbles to the slightly subcritical problem with the fractional Laplacian
\begin{equation*}
\left\{\aligned
&(-\Delta)^su=|u|^{p-1-\varepsilon}u\ \ \mbox{in}\ \Omega,\\
&u=0\ \ \ \ \ \ \ \ \ \ \ \ \ \ \ \ \ \ \ \ \ \mbox{on}\ \partial\Omega,
\endaligned \right.
\end{equation*}
where  $\Omega$ is a smooth bounded domain in $\mathbb R^N$,  $N>2s$, $0<s<1$, $ p=\frac{N+2s}{N-2s}$ and $\varepsilon>0$ is a small parameter, which can be seen as a nonlocal analog of the results of Bartsch, Micheletti and Pistoia (2006) \cite{Bartsch1}.
\end{abs}

\noindent
\begin{key words}
Fractional Laplacian, Nodal solutions, Slightly subcritical problem, Lyapunov-Schmidt reduction
\end{key words}
\indent
\section{\textbf{Introduction}\label{Section 1}}
This paper is devoted to the problem involving Fractional Laplacian
\begin{equation}\label{pro}
\left\{\aligned
&(-\Delta)^su=|u|^{p-1-\varepsilon}u\ \ \ \mbox{in}\ \Omega,\\
&u=0\ \ \ \ \ \ \ \ \ \ \ \ \ \ \ \ \ \ \ \ \ \ \mbox{on}\ \partial\Omega,
\endaligned \right.
\end{equation}
where $\Omega$ is a smooth bounded domain in $\mathbb R^N$,  $N>2s$, $0<s<1$, $ p=\frac{N+2s}{N-2s}$ and $\varepsilon>0$ is a small parameter, $(-\Delta)^s$ stands for the fractional Laplacian operator.

The fractional Laplacian appears in physics, biological modeling, probability and mathematical finance, which is a nonlocal operator. Therefore it is difficult to handle and has attracted much attention in recent years. Importantly, Caffarelli, Silvestre \cite{Caffarelli2} developed an extension method to transform the nonlocal problem into a local one, which helps to study the fractional Laplacian by purely local arguments. By using their extension, many authors studied the existence of solutions to problem $(-\Delta)^su=f(u)$ with $f: \mathbb R^N\rightarrow \mathbb R$. For example, when $s=\frac{1}{2}$, Cabr\'{e} and Tan \cite{Tan17} and Tan \cite{Jing13} established the existence of positive solutions for nonlinear equations having subcritical growth. 

Then it is interesting to study the blow-up phenomenon of solutions to (\ref{pro}) as $\epsilon \to 0^+$. For positive solutions, Chio, Kim and Lee \cite{Chio12} established the asymptotic behavior of least energy solutions and the existence of multiple bubbling solutions. Rois and Luis \cite{Rois15} generalized the work of Chio, Kim and Lee \cite{Chio12}, and took into account both subcritical and supercritical case. These papers are however not concerned with the nodal solutions involving the Fractional Laplacian.

If $s=1$, problem (\ref{pro}) was extensively studied about the blow-up phenomenon of positive and nodal solutions. It was proved in \cite{Rey3,Han4,Brezis5} that as $\varepsilon$ goes to zero, positive solution $u$ to problem (\ref{pro}) blows up and concentrates at a critical point of the Robin's function. Rey \cite{Rey6} considered the positive solutions with double blow-up and showed that the two concentration points $\sigma_1^*$ and $\sigma_2^*$ must be such that $(\sigma_1^*, \sigma_2^*)$ is a critical point of the function
\begin{equation}
\Phi(\sigma_1, \sigma_2)=H^\frac{1}{2}(\sigma_1, \sigma_1)H^\frac{1}{2}(\sigma_2, \sigma_2)-G(\sigma_1, \sigma_2),\ \ \ \ (\sigma_1, \sigma_2)\in\Omega\times\Omega
\end{equation}
and satisfies $\Phi(\sigma_1^*, \sigma_2^*)\geq0$. Here $G$ is the Green's function of the Dirichlet Laplacian and $H$ is its regular part. See also \cite{Bahri7} for the existence of positive solutions with multiple bubbles. In a convex domain, it proved in \cite{Grossi8} that no
positive solutions have multiple bubbles for problem (\ref{pro}). On the other hand, nodal solutions with multiple-bubbles also exist for the problem (\ref{pro}) with $s=1$ in a general smooth bounded domain $\Omega$. As the parameter $\varepsilon$ goes to zero, Bartsch, Micheletti and Pistoia \cite{Bartsch1} proved the existence of nodal solutions which blow up
positively at a point $\sigma_1^*\in\Omega$ and blow up negatively at a point $\sigma_2^*\in\Omega$, with $\sigma_1^*\neq\sigma_2^*$. In \cite{Ben9} Ben Ayed, Mehdi and Pacella classified the nodal solutions according to the concentration speeds of the positive and negative part. Bartsch, D'Aprile and Pistoia in \cite{Bartsch10}
studied the existence of nodal solutions with four bubbles in a smooth bounded domain $\Omega$. When $\Omega$ is a ball, they also proved the nodal with three bubbles in \cite{Bartsch11}.

In this paper, we are interested in the existence of nodal solutions which blow-up and concentrate at two different points of the domain $\Omega$.

In order to state our result,  we introduce some well known notations. Let $G$ be the Green's function of $(-\Delta)^s$ in $\Omega$ with Dirichlet boundary conditions, that is, $G$ satisfies
\begin{equation}\label{Green}
\left\{\aligned
&(-\Delta)^sG(\cdot,y)=\delta_y\ \ \ \mbox{in}\ \Omega,\\
&G(\cdot,y)=0\ \ \ \ \ \ \ \ \ \ \ \ \ \mbox{on}\ \partial\Omega,
\endaligned \right.
\end{equation}
where $\delta_y$ denotes the Dirac mass at the point $y$. The regular part of $G$ is given by

\begin{equation}\label{Green-regular}
H(x,y)=\frac{c_{N,s}}{|x-y|^{N-2s}}-G(x,y)\ \ \ \mbox{where}\ \ c_{N,s}=\frac{2^{1-2s}\Gamma(\frac{N-2s}{2})}{2\pi^{\frac{N}{2}}\Gamma(s)}.
\end{equation}
The diagonal $H(x,x)$ is usually called the Robin's function of the domain $\Omega$.

Now we can state the main result. Let us consider
the function $\varphi:\Omega\times\Omega\rightarrow\mathbb R$ defined by
\begin{equation}\label{phi-function}
\varphi(\sigma_1,\sigma_2)=H^\frac{1}{2}(\sigma_1, \sigma_1)H^\frac{1}{2}(\sigma_2, \sigma_2)+G(\sigma_1, \sigma_2),
\end{equation}
which will play a crucial role in our analysis.

\begin{theorem}\label{theorem-1 bubbles}
Suppose that $0<s<1$ and $N>2s$, then there exists a small number $\varepsilon_0>0$ such that for $0<\varepsilon<\varepsilon_0$, problem (\ref{pro}) has a pair of solutions $u_\varepsilon$ and $-u_\varepsilon$. As $\varepsilon$ goes to zero, $u_\varepsilon$ blows up positively at a point $\sigma_1^*\in\Omega$ and negatively at a point $\sigma_2^*\in\Omega$, where $\varphi(\sigma_1^*,\sigma_2^*)=\min\limits_{\Omega\times\Omega}\varphi$.
\end{theorem}

The proof of Theorem \ref{theorem-1 bubbles} is motivated the result of Bartsch, Micheletti and Pistoia \cite{Bartsch1} on the local problem, based on a Lyapunov-Schmidt reduction scheme. The main point is to find critical points of the finite dimensional reduced functional corresponding to critical points of the energy function of problem (\ref{pro}). The reduced functional is given in terms of the Green's and Robin's functions. In
subcritical case, the role of Green's and Robin's functions in the concentration phenomena associated to the critical exponent has already been considered in several works, see \cite{Rey3,Han4,Brezis5,Chio12} and \cite{Rey6,Bahri7}. The proofs here borrow ideas of the above mentioned works.

This paper is organized as follows. In section 2, we present some definitions and the basic properties of the fractional Laplacian in bounded domains and in the whole $\mathbb R^N$. Section 3 is devoted to developing the analytical tools toward the main results. Moreover, nodal solutions are constructed by the Lyapunov-Schmidt reduction method. Finally, in the Appendix, some necessary estimates for the construction of the nodal solutions are exhibited.

\section{\textbf{Preliminary }\label{Section 2}}

In this section we review some basic definitions and properties of the fractional Laplacian. We refer to \cite{Colorado16,Tan17,Capella18,Kim19,Tan20,Stinga21,Barrios22} for the details.

Let $\Omega$ be a smooth bounded domain in $\mathbb R^N$. We define $(-\Delta)^s$ through the spectral decomposition using the powers of the eigenvalues of the Laplacian operator $-\Delta$ in $\Omega$. Let $\{\lambda_i,\phi_i\}_{i=1}^\infty$ denote the eigenvalues and eigenfunctions of $-\Delta$ in $\Omega$ with zero Dirichlet boundary condition,

\begin{equation}\label{Dirichlet}
\left\{\aligned
&-\Delta\phi_i=\lambda_i\phi_i\ \ \ \ \mbox{in}\ \Omega,\\
&\phi_i=0\ \ \ \ \ \ \ \ \ \ \ \ \ \ \mbox{on}\ \partial\Omega.
\endaligned \right.
\end{equation}
 The fractional Laplacian is well defined in the fractional Sobolev space $H_{0}^s(\Omega)$,
\[H_{0}^s(\Omega):=\left\{u=\sum_{i=1}^\infty a_i\phi_i \in L^2(\Omega):\sum_{i=1}^\infty a_{i}^2\lambda_{i}^s<\infty\right\},
\]
which is a Hilbert space whose inner product is defined by
\[\left\langle \sum_{i=1}^\infty a_i\phi_i, \sum_{i=1}^\infty b_i\phi_i\right\rangle_{H_{0}^s(\Omega)}=\sum_{i=1}^\infty a_i b_i\lambda_{i}^s.
\]
Moreover, we define fractional Laplacian $(-\Delta)^s:H_{0}^s(\Omega)\rightarrow H_{0}^s(\Omega)$ as:
\[(-\Delta)^s\left(\sum_{i=1}^\infty a_i\phi_i\right)=\sum_{i=1}^\infty a_i \lambda_{i}^s\phi_i.
\]
Note that by the above definitions, we have the following expression for the inner product:
\begin{equation}\label{inner product}
\langle u,v\rangle_{H_{0}^s(\Omega)}=\int_\Omega(-\Delta)^{\frac{s}{2}}u(-\Delta)^{\frac{s}{2}}v=\int_\Omega(-\Delta)^suv,\ \ \ \ u,v \in H_{0}^s(\Omega).
\end{equation}

We will recall an equivalent definition based on an extension problem introduced by Caffarelli and Silvestre \cite{Caffarelli2}. For the sake of simplicity, we denote  $\Omega\times(0,\infty)$ by $\mathcal{C}$ and its lateral boundary $\partial\Omega\times(0,\infty)$ by $\partial_\mathcal{C}$, where $\Omega$ is either a smooth bounded domain or $\mathbb R^N$. If $\Omega$ is a smooth bounded domain, the function space $H_{0,L}^s(\mathcal{C})$ is defined as the completion of

\[C_{\mathcal{C},L}^\infty(\mathcal{C}):=\left\{U\in C^\infty(\bar{\mathcal{C}}):U=0\ \mbox{on}\ \partial_L \mathcal{C}\right\}
\]
with respect to the norm
\begin{equation}\label{norm}
\parallel U\parallel_\mathcal{C}=\left(\frac{1}{k_s}\int_\mathcal{C} t^{1-2s}|\nabla U|^2\right)^\frac{1}{2},
\end{equation}
where $t>0$ represents the last variable in $\mathbb R^{N+1}$ and $k_s$ is a normalization constant (see \cite{Barrios22,Chio12}).
This is a Hilbert space endowed with the following inner product
\[(U,V)_\mathcal{C}=\frac{1}{k_s}\int_\mathcal{C}t^{1-2s}\nabla U\cdot\nabla V
\quad\mbox{for\ all}\ \ U,V\in H_{0,L}^s(\mathcal{C}).
\]
Moreover, in the entire space, we define $\mathcal{D}^s(\mathbb R_{+}^{N+1})$ as the completion of $C_{\mathcal{C}}^\infty(\overline{\mathbb R_{+}^{N+1}})$ with respect to the norm $\parallel U\parallel_{\mathbb R_{+}^{N+1}}$ (defined as in (\ref{norm}) by putting $\mathcal{C}=\mathbb R_{+}^{N+1}$). We know that if $\Omega$ is a smooth bounded domain, then
\begin{equation}\label{trace}
H_{0}^s(\Omega)=\left\{u=\mbox{tr}|_{\Omega\times\{0\}}U:U\in H_{0,L}^s(\mathcal{C})\right\}.
\end{equation}
It also holds that
\[\parallel U(\cdot,0)\parallel_{H^s(\mathbb R^N)}\leq C\parallel U\parallel_{\mathbb R_{+}^{N+1}}
\]
for some $C>0$ independent of $U\in \mathcal{D}^s(\mathbb R_{+}^{N+1})$.

Now we consider the harmonic extension problem. For some given functions $u\in H_{0}^s(\Omega)$, $U\in H_{0,L}^s(\mathcal{C})$ solves the equation
\begin{equation}\label{extension}
\left\{\aligned
&\mbox{div}(t^{1-2s}\nabla U)=0\ \ \ \ \ \mbox{in}\ \mathcal{C},\\
&U=0\ \ \ \ \ \ \ \ \ \ \ \ \ \ \ \ \ \ \ \ \mbox{on}\ \partial_L \mathcal{C},\\
&U(x,0)=u(x)\ \ \ \ \ \ \ \ \ \mbox{on}\ \Omega\times\{0\}
\endaligned \right.
\end{equation}
as a unique solution.
The relevance of the extension function $U$ is that it is related to the fractional Laplacian of the original function $u$ through the formula
\begin{equation}\label{formula}
-\frac{1}{k_s}\lim\limits_{t\rightarrow0^{+}}t^{1-2s}\frac{\partial U}{\partial t}(x,t)=(-\Delta)^su(x),
\end{equation}
where $k_s>0$ depends on $N$ and $s$ (see \cite{Caffarelli2} and \cite{Capella18} for the entire and bounded domain case, respectively). By the above extension, the problem (\ref{pro}) is transformed into its equivalence problem
\begin{equation}\label{equi}
\left\{\aligned
&\mbox{div}(t^{1-2s}\nabla U)=0\ \ \ \ \ \ \ \ \ \ \ \ \ \ \ \  \ \ \ \ \ \ \ \ \ \ \ \ \ \ \ \mbox{in}\ \mathcal{C},\\
&U=0\ \ \ \ \ \ \ \ \ \ \ \ \ \ \ \ \ \ \ \ \ \ \ \ \ \ \ \ \  \ \ \ \ \ \ \ \ \ \ \ \ \ \ \ \ \ \mbox{on}\ \partial_L \mathcal{C},\\
&-\frac{1}{k_s}\lim\limits_{t\rightarrow0^{+}}t^{1-2s}\frac{\partial U}{\partial t}(x,t)=|U|^{p-1-\varepsilon}U\ \ \ \ \mbox{on}\ \Omega\times\{0\}.
\endaligned \right.
\end{equation}

In a completely analogous extension procedure, the Green's function $G$ of the fractional Laplacian $(-\Delta)^s$ defined in (\ref{Green}) can be regarded as the trace of the solution $G_\mathcal{C}(z,y)$ ($z=(x,t)\in \mathcal{C}$, $y\in\Omega$) for the following extended Dirichlet-Neumann problem
\begin{equation}
\left\{\aligned
&\mbox{div}(t^{1-2s}\nabla G_\mathcal{C}(\cdot,y))=0\ \ \ \ \ \ \ \ \ \ \ \ \ \mbox{in}\ \mathcal{C},\\
&G_\mathcal{C}(\cdot,y)=0\ \ \ \ \ \ \ \ \ \ \ \ \ \ \ \ \ \ \ \ \ \ \ \ \ \ \ \ \mbox{on}\ \partial_L \mathcal{C},\\
&-\frac{1}{k_s}\lim\limits_{t\rightarrow0^{+}}t^{1-2s}\frac{\partial G_\mathcal{C}(\cdot,y)}{\partial t}=\delta_y\ \ \ \ \mbox{on}\ \Omega\times\{0\}.
\endaligned \right.
\end{equation}
Moreover, if a function $U$ satisfies
\begin{equation}
\left\{\aligned
&\mbox{div}(t^{1-2s}\nabla U)=0\ \ \ \ \ \ \ \ \ \ \ \ \ \ \ \ \ \ \ \ \ \ \ \mbox{in}\ \mathcal{C},\\
&U=0\ \ \ \ \ \ \ \ \ \ \ \ \ \ \ \ \ \ \ \ \ \ \ \ \ \ \ \ \ \ \ \ \ \ \ \ \ \ \mbox{on}\ \partial_L \mathcal{C},\\
&-\frac{1}{k_s}\lim\limits_{t\rightarrow0^{+}}t^{1-2s}\frac{\partial U_\mathcal{C}(\cdot,y)}{\partial t}=g(x)\ \ \ \ \mbox{on}\ \Omega\times\{0\},
\endaligned \right.
\end{equation}
then we have the following expression
\[U(z)=\int_\Omega G_\mathcal{C}(z,y)g(y)dy=\int_\Omega G_\mathcal{C}(z,y)(-\Delta)^su(y)dy\ \ \ \ \ \ \mbox{for all}\ \ z\in \mathcal{C},
\]
where $u=\mbox{tr}|_{\Omega\times\{0\}}U$.

Green's function $G_\mathcal{C}$ can be partitioned to the singular part and the regular part on $\mathcal{C}$. For the singular part of the Green's function $G_\mathcal{C}$, it can be given by
\begin{equation}
G_{\mathbb R_{+}^{N+1}}(z,y):=\frac{c_{N,s}}{|y-z|^{N-2s}},
\end{equation}
where $c_{N,s}$ is defined in (\ref{Green-regular}), and $G_{\mathbb R_{+}^{N+1}}$ solves the problem
$$\aligned\quad
\left\{\aligned
&\mbox{div}(t^{1-2s}\nabla G_{\mathbb R_{+}^{N+1}}(z,y))=0\ \ \ \ \ \ \ \ \ \ \ \ \ \mbox{in}\ \mathbb  R_{+}^{N+1},\\
&-\frac{1}{k_s}\lim\limits_{t\rightarrow0^{+}}t^{1-2s}\frac{\partial G_{\mathbb R_{+}^{N+1}}(z,y)}{\partial t}=\delta_y\ \ \ \ \mbox{on}\ \Omega\times\{0\}
\endaligned \right.
\endaligned$$
for $y\in \mathbb R^N$.
The regular part can be seen as the unique solution to
$$\aligned\quad
\left\{\aligned
&\mbox{div}(t^{1-2s}\nabla H_\mathcal{C}(z,y))=0\ \ \ \ \ \ \ \ \ \ \mbox{in}\ \mathcal{C},\\
&H_\mathcal{C}(z,y)=\frac{c_{N,s}}{|y-z|^{N-2s}}\ \ \ \ \ \ \ \ \ \ \ \mbox{on}\ \partial_L \mathcal{C},\\
&-\lim\limits_{t\rightarrow0^{+}}t^{1-2s}\frac{\partial H_\mathcal{C}(z,y)}{\partial t}=0\ \ \ \ \ \ \mbox{on}\ \Omega\times\{0\}.
\endaligned \right.
\endaligned$$
Then we have
\begin{equation}\label{regular}
G_\mathcal{C}(z,y)=G_{\mathbb R_{+}^{N+1}}(z,y)-H_\mathcal{C}(z,y).
\end{equation}

Next we present the Sharp Sobolev and trace inequalities (see \cite{Chio12,Cotsiolis23}). Given any $\lambda>0$ and $\xi\in\mathbb R^N$, here
\begin{equation}\label{solution}
w_{\lambda,\xi}(x)=a_{N,s}\left(\frac{\lambda}{\lambda^2+|x-\xi|^2}\right)^\frac{N-2s}{2}
\end{equation}
is an explicit family of solutions to
\begin{equation}\label{critical}
(-\Delta)^su=u^p\ \ \ \ \ \mbox{in}\ \mathbb R^N,\\
\end{equation}
where
$a_{N,s}>0$ (see \cite{Chen24,Li25,Li26} for details).
Then the sharp Sobolev inequality from \cite{Cotsiolis23} is the following:
\begin{equation}\label{Sobolev}
\left(\int_{\mathbb R^N}|u|^{p+1}dx\right)^\frac{1}{p+1}\leq S_{N,s}\left(\int_{\mathbb R^N}|(-\Delta)^\frac{s}{2}u|^2dx\right)^\frac{1}{2}.
\end{equation}
The equality is attained if and only if $u(x)=cw_{\lambda,\xi}(x)$ for any $c>0$, $\lambda>0$ and $\xi\in\mathbb R^N$, where
\[S_{N,s}=2^{-s}\pi^{-s/2}\left[\frac{\Gamma(\frac{N-2s}{2})}{\Gamma(\frac{N+2s}{2})}\right]^\frac{1}{2}\left[\frac{\Gamma(N)}{\Gamma(N/2)}\right]^\frac{s}{N},
\]
(refer to \cite{Carlen27,Frank28,Lieb29}). Now let $W_{\lambda,\xi}\in \mathcal{D}^s(\mathbb R_{+}^{N+1})$ be the s-harmonic extension of $w_{\lambda,\xi}$ satisfying
\begin{equation}\label{s-harmonic}
\left\{\aligned
&\mbox{div}(t^{1-2s}\nabla W_{\lambda,\xi}(x,t))=0\ \ \ \ \  \ \ \ \ \mbox{in}\ \ \mathbb R_{+}^{N+1},\\
&W_{\lambda,\xi}(x,0)=w_{\lambda,\xi}(x)\ \ \ \ \ \ \  \ \ \ \  \ \ \ \ \ \mbox{for}\ x\in\mathbb  R^N.\\
\endaligned \right.
\end{equation}
It implies that the Sobolev trace inequality
\begin{equation}\label{inequality}
\left(\int_{\mathbb R^N}|U(x,0)|^{p+1}dx\right)^\frac{1}{p+1}\leq \frac{S_{N,s}}{\sqrt{k_s}}\left(\int_0^\infty\int_{\mathbb R^N}t^{1-2s}|\nabla U(x,t)|^2dxdt\right)^\frac{1}{2}
\end{equation}
gets the equality if and only if $U(x,t)=cW_{\lambda,\xi}(x,t)$ for any $c>0$, $\lambda>0$ and $\xi\in\mathbb R^N$, where $k_s>0$ is given in (\ref{formula}) (see \cite{Xiao30}).\\

\section{\textbf{The finite dimensional reduction}\label{Section 3}}

In this section we are devoted to proving Theorem \ref{theorem-1 bubbles} by applying the Lyapunov-Schmidt reduction method. Similar methods are used in \cite{Bartsch1,Rey6,Bahri7,Musso31}.

Let $\Omega$ be a smooth bounded domain in $\mathbb R^N$. Set
\begin{equation}\label{set}
\Omega_\varepsilon=\varepsilon^{-\frac{1}{N-2s}}\Omega=\{\varepsilon^{-\frac{1}{N-2s}}x:\ \ x\in\Omega\},
\end{equation}
then the changing variables
\begin{equation}\label{variables}
v(x)=\varepsilon^\frac{1}{2-\frac{\varepsilon(N-2s)}{2s}}u(\varepsilon^\frac{1}{N-2s}x)\ \ \ \ \mbox{for}\ \ x\in\Omega_\varepsilon
\end{equation}
transforms equation (\ref{pro}) into
\begin{equation}\label{transform}
\left\{\aligned
&(-\Delta)^sv=|v|^{p-1-\varepsilon}v\ \ \ \mbox{in}\ \Omega_\varepsilon,\\
&v=0\ \ \ \ \ \ \ \ \ \ \ \ \ \ \ \ \ \ \ \ \ \ \mbox{on}\ \partial\Omega_\varepsilon.
\endaligned \right.
\end{equation}
It follows that $u(x)$ is a solution to (\ref{pro}) if and only if $v(x)=\varepsilon^\frac{1}{2-\frac{\varepsilon(N-2s)}{2s}}u(\varepsilon^\frac{1}{N-2s}x)$ is a solution of equation (\ref{transform}). In the proof of Theorem 1.1, solutions to (\ref{transform}) are close related to the following dilated equation
\begin{equation}\label{dilated equation}
\left\{\aligned
&\mbox{div}(t^{1-2s}\nabla V)=0\ \ \ \ \ \ \ \ \ \ \ \ \ \ \ \  \ \ \ \ \ \ \ \ \ \ \ \ \ \ \mbox{in}\ \mathcal{C}_\varepsilon,\\
&V=0\ \ \ \ \ \ \ \ \ \ \ \ \ \ \ \ \ \ \ \ \ \ \ \ \ \ \ \ \  \ \ \ \ \ \ \ \ \ \ \ \ \ \ \ \ \mbox{on}\ \partial_L \mathcal{C}_\varepsilon,\\
&-\frac{1}{k_s}\lim\limits_{t\rightarrow0^{+}}t^{1-2s}\frac{\partial V}{\partial t}=|V|^{p-1-\varepsilon}V\ \ \ \ \ \ \ \ \ \mbox{on}\ \Omega_\varepsilon\times\{0\}.
\endaligned \right.
\end{equation}
It is easy to know that if $V$ is the solution of (\ref{dilated equation}), then $U(x)=\varepsilon^{-\frac{1}{2-\frac{\varepsilon(N-2s)}{2s}}}V(\varepsilon^{-\frac{1}{N-2s}}x)$ solve the problem (\ref{equi}). To look for the solutions that satisfy the equation (\ref{equi}), it suffices to apply the Lyapunov-Schmidt reduction method to the extended problem (\ref{dilated equation}).
Moreover, it is easy to know that the harmonic extension $V$ of function $v$ satisfies the problem
\begin{equation}\label{harmonic extension}
\left\{\aligned
&\mbox{div}(t^{1-2s}\nabla V)=0\ \ \ \ \mbox{in}\ \mathcal{C}_\varepsilon,\\
&V=0\ \ \ \ \ \ \ \ \ \ \ \ \ \ \ \ \ \ \ \mbox{on}\ \partial_L \mathcal{C}_\varepsilon,\\
&V(x,0)=v(x)\ \ \ \ \ \ \ \ \mbox{on}\ \Omega_\varepsilon\times\{0\},
\endaligned \right.
\end{equation}
where $\mathcal{C}_\varepsilon=\varepsilon^{-\frac{1}{N-2s}}\mathcal{C}=\{\varepsilon^{-\frac{1}{N-2s}}(x,t):\ \ (x,t)\in\mathcal{C}\}$.

Let us recall the the functions $w_{\lambda,\xi}$ and $W_{\lambda,\xi}$ defined in (\ref{solution}) and (\ref{s-harmonic}). By the result of \cite{del32}, it is known that the kernel of the operator $(-\Delta)^s-pw_{\lambda,\xi}^{p-1}$ is spanned by the functions
\begin{equation}\label{span}
\frac{\partial w_{\lambda,\xi}}{\partial\xi_1},\ \cdot\cdot\cdot,\frac{\partial w_{\lambda,\xi}}{\partial\xi_N}\ \ \mbox{and}\ \ \frac{\partial w_{\lambda,\xi}}{\partial\lambda},
\end{equation}
namely they satisfy the equation
\begin{equation}\label{namely}
(-\Delta)^s\phi=pw_{\lambda,\xi}^{p-1}\phi \ \ \ \ \mbox{in}\ \mathbb R^N,
\end{equation}
where $\xi=(\xi_1,\cdot\cdot\cdot\cdot\cdot,\xi_N)$ in $\mathbb R^N$. We also have that all bounded solutions of the extended problem of (\ref{namely})
\begin{equation}\label{bounded solutions}
\left\{\aligned
&\mbox{div}(t^{1-2s}\nabla \Phi)=0\ \ \ \ \ \ \ \ \ \ \ \ \ \ \ \  \ \ \ \ \ \ \ \mbox{in}\ \mathbb R_{+}^{N+1},\\
&-\frac{1}{k_s}\lim\limits_{t\rightarrow0^{+}}t^{1-2s}\frac{\partial \Phi}{\partial t}=pw_{\lambda,\xi}^{p-1}\Phi\ \ \ \ \ \ \mbox{on}\ \mathbb R^N\times\{0\}
\endaligned \right.
\end{equation}
consist of the linear combinations of the functions
\begin{equation}\label{linear combinations}
\frac{\partial W_{\lambda,\xi}}{\partial\xi_1},\ \ \cdot\cdot\cdot,\frac{\partial W_{\lambda,\xi}}{\partial\xi_N}\ \ \mbox{and}\ \ \frac{\partial W_{\lambda,\xi}}{\partial\lambda}.
\end{equation}
In order to construct the multi-bubble nodal solutions of (\ref{equi}), for $\eta\in(0,1)$, we define the admissible set
\begin{align}\label{admissible set}
&\mathcal{O}_\eta=\{(\mbox{\boldmath $\lambda$},\mbox{\boldmath $\sigma$}):=((\lambda_1,...,\lambda_k), (\sigma_1,...,\sigma_k))\in\mathbb R_{+}^{k}\times\Omega^k,
\ \ \sigma_i=(\sigma_{i}^1,...,\sigma_{i}^N),\nonumber\\
&\mbox{dist}(\sigma_i,\partial\Omega)>\eta, \ \ \eta<\lambda_i<\frac{1}{\eta}, \ \ |\sigma_i-\sigma_j|>\eta, \ \ i\neq j, \ \ i,j=1,...,k \}.
\end{align}

It is useful to rewrite problem (\ref{pro}) in a different setting. To this end, let us introduce the following operator.

\begin{definition}\label{definintion-operator L-1}
Let the map
\begin{equation}\label{map}
i_{\varepsilon}^*:L^\frac{2N}{N+2s}(\Omega_\varepsilon)\rightarrow H_{0,L}^s(\mathcal{C}_\varepsilon)
\end{equation}
be the adjoint operator of the Sobolev trace embedding
\begin{equation}
i_{\varepsilon}:H_{0,L}^s(\mathcal{C}_\varepsilon)\rightarrow L^\frac{2N}{N-2s}(\Omega_\varepsilon)\nonumber
\end{equation}
defined by the
\begin{equation}
i_{\varepsilon}(V)=\mbox{tr}|_{\Omega_\varepsilon\times\{0\}}(V)\nonumber\ \ \ \ \ \mbox{for}\ \ V\in H_{0,L}^s(\mathcal{C}_\varepsilon),
\end{equation}
which comes from the inequality (\ref{inequality}), that is, for some $v\in L^\frac{2N}{N+2s}(\Omega_\varepsilon)$ and $Z\in H_{0,L}^s(\mathcal{C}_\varepsilon)$,
\begin{equation}
i_{\varepsilon}^*(v)=Z\nonumber
\end{equation}
if and only if
\begin{equation}
\left\{\aligned
&\mbox{div}(t^{1-2s}\nabla Z)=0\ \ \ \ \ \ \ \ \ \ \ \ \mbox{in}\ \mathcal{C}_\varepsilon,\\
&Z=0\ \ \ \ \ \ \ \ \ \ \ \ \ \ \ \ \ \ \ \ \ \ \ \ \ \ \mbox{on}\ \partial_L \mathcal{C}_\varepsilon,\\
&-\frac{1}{k_s}\lim\limits_{t\rightarrow0^{+}}t^{1-2s}\frac{\partial Z}{\partial t}=v\ \ \ \ \mbox{on}\ \Omega_\varepsilon\times\{0\}.\nonumber
\endaligned \right.
\end{equation}
\end{definition}
By the definition of the operator $i_{\varepsilon}^*$, solving problem (\ref{dilated equation}) is equivalent to find a solution of the fixed point problem
\begin{equation}
V=k_s i_{\varepsilon}^*\left(f_\varepsilon\left(i_{\varepsilon}(V)\right)\right),\ \ \ V\in H_{0,L}^s(\mathcal{C}_\varepsilon),
\end{equation}
where $f_\varepsilon(s)=|s|^{p-1-\varepsilon}s$. Notice that from (\ref{trace}) we have $i_{\varepsilon}:H_{0,L}^s(\mathcal{C}_\varepsilon)\rightarrow H_{0}^s(\Omega_\varepsilon)\subset L^\frac{2N}{N-2s}(\Omega_\varepsilon)$ and so $(-\Delta)^s(i_{\varepsilon}(U))$ makes sense.

We look for solutions of (\ref{transform}) of the form
\begin{equation}
v=\sum_{i=1}^ka_i\mathcal{P}_\varepsilon w_{\lambda_i,\delta_i}+\phi_\varepsilon,\nonumber
\end{equation}
for $k\geq1$ a fixed integer and $a_1,...,a_k\in\{\pm1\}$ fixed, where $\phi_\varepsilon$ is a lower order term and $\mathcal{P}_\varepsilon: H^s(\mathbb R^N)\rightarrow H_{0}^s(\Omega_\varepsilon)$ is the projection defined by the equation
\begin{equation}
\left\{\aligned
&(-\Delta)^s\mathcal{P}_\varepsilon w_i=(-\Delta)^sw_i \ \ \ \mbox{in}\ \Omega_\varepsilon,\\
&\mathcal{P}_\varepsilon w_i=0  \ \ \ \ \ \ \ \ \ \ \ \ \ \ \ \ \ \ \ \ \ \ \mbox {on} \ \partial\Omega_\varepsilon,\nonumber
\endaligned \right.
\end{equation}
where $w_i=w_{\lambda_i,\delta_i}$, $\delta_i=\varepsilon^{-\frac{1}{N-2s}}\sigma_i\in\Omega_\varepsilon$.

Let us introduce some notations. For $\xi=(\xi^1,...,\xi^N)\in\mathbb R^N$ and $j=1,2,...,N$, we define the functions
\begin{equation}\label{functions}
\Psi_{\lambda,\xi}^0=\frac{\partial W_{\lambda,\xi}}{\partial\lambda},\ \ \Psi_{\lambda,\xi}^j=\frac{\partial W_{\lambda,\xi}}{\partial\xi^j},\ \ \psi_{\lambda,\xi}^0=\frac{\partial w_{\lambda,\xi}}{\partial\lambda},\ \ \psi_{\lambda,\xi}^j=\frac{\partial w_{\lambda,\xi}}{\partial\xi^j}
\end{equation}
and
\begin{equation}\label{notations}
\mathcal{P}_\varepsilon W_{\lambda,\xi}=i_{\varepsilon}^*(w_{\lambda,\xi}^p),\ \ \mathcal{P}_\varepsilon \Psi_{\lambda,\xi}^j=i_{\varepsilon}^*(pw_{\lambda,\xi}^{p-1}\psi_{\lambda,\xi}^j),\ \ \ \ j=0,1,...,N.
\end{equation}
Moreover, we let the functions $\mathcal{P}_\varepsilon w_{\lambda,\xi}$ and $\mathcal{P}_\varepsilon \psi_{\lambda,\xi}^j$ be
\begin{equation}\label{define}
\mathcal{P}_\varepsilon w_{\lambda,\xi}=i_{\varepsilon}(\mathcal{P}_\varepsilon W_{\lambda,\xi}),\ \ \mathcal{P}_\varepsilon \psi_{\lambda,\xi}^j=i_{\varepsilon}(\mathcal{P}_\varepsilon\Psi_{\lambda,\xi}^j),\ \ \ \ \ j=0,1,...,N
\end{equation}
which satisfy the equations $(-\Delta)^su=w_{\lambda,\xi}^p$ and $(-\Delta)^su=pw_{\lambda,\xi}^{p-1}\psi_{\lambda,\xi}^j$ in $\Omega_\varepsilon$, respectively. For the sake of simplicity, we denote
\begin{equation}\label{simplicity}
W_i=W_{\lambda_i,\delta_i},\
\mathcal{P}_\varepsilon W_i=\mathcal{P}_\varepsilon W_{\lambda_i,\delta_i},\ \Psi_i^j=\Psi_{\lambda_i,\delta_i}^j,\ \mathcal{P}_\varepsilon\Psi_i^j=\mathcal{P}_\varepsilon\Psi_{\lambda_i,\delta_i}^j,\ i=1,2,...,k,\ \ j=0,1,...,N.
\end{equation}
for $\delta_i=\varepsilon^{-\frac{1}{N-2s}}\sigma_i\in\Omega_\varepsilon$ and $(\mbox{\boldmath $\lambda$},\mbox{\boldmath $\sigma$})\in\mathcal{O}_\eta$.
Similary, we denote
\begin{equation}\label{Similary}
\mathcal{P}_\varepsilon w_i=\mathcal{P}_\varepsilon w_{\lambda_i,\delta_i},\ \ \mathcal{P}_\varepsilon\psi_i^j=\mathcal{P}_\varepsilon\psi_{\lambda_i,\delta_i}^j,\ \ \ \ \ \ i=1,2,...,k,\ \ j=0,1,...,N.
\end{equation}
Set the space
\begin{equation}\label{space}
\mathcal{K}_{\mbox{\boldmath $\lambda$},\mbox{\boldmath $\sigma$}}^\varepsilon=\{u\in H_{0,L}^1(\mathcal{C}_\varepsilon):\ \ (u,\mathcal{P}_\varepsilon\psi_i^j)_{\mathcal{C}_\varepsilon}=0 ,\ \ i=1,2,...,k,\ \ j=0,1,...,N\},
\end{equation}
where $\varepsilon>0$ and $(\mbox{\boldmath $\lambda$},\mbox{\boldmath $\sigma$})\in\mathcal{O}_\eta$. We also need the following orthogonal projections
\begin{equation}\label{projection}
\Pi_{\mbox{\boldmath $\lambda$},\mbox{\boldmath $\sigma$}}^\varepsilon:\ \ H_{0,L}^s(\mathcal{C}_\varepsilon)\rightarrow\mathcal{K}_{\mbox{\boldmath $\lambda$},\mbox{\boldmath $\sigma$}}^\varepsilon.
\end{equation}

Now if we let the linear operator $L_{\mbox{\boldmath $\lambda$},\mbox{\boldmath $\sigma$}}^\varepsilon:\ \mathcal{K}_{\mbox{\boldmath $\lambda$},\mbox{\boldmath $\sigma$}}^\varepsilon\rightarrow\mathcal{K}_{\mbox{\boldmath $\lambda$},\mbox{\boldmath $\sigma$}}^\varepsilon$ be defined by
\begin{equation}\label{operator}
L_{\mbox{\boldmath $\lambda$},\mbox{\boldmath $\sigma$}}^\varepsilon(\Phi)=\Phi-\Pi_{\mbox{\boldmath $\lambda$},\mbox{\boldmath $\sigma$}}^\varepsilon i_{\varepsilon}^*\left[f'_0(\sum_{i=1}^ka_i\mathcal{P}_\varepsilon w_{\lambda_i,\delta_i})\cdot i_{\varepsilon}(\Phi)\right],
\end{equation}
then we can give an a-prior estimate for $\Phi\in\mathcal{K}_{\mbox{\boldmath $\lambda$},\mbox{\boldmath $\sigma$}}^\varepsilon$.

\begin{lemma}\label{lemma-estimate of error-1}
For any $\eta>0$ there exists sufficiently small $\varepsilon>0$ and a constant $C=C(N,\eta)$ such that, for every $(\mbox{\boldmath $\lambda$},\mbox{\boldmath $\sigma$})\in\mathcal{O}_\eta$, the operator $L_{\mbox{\boldmath $\lambda$},\mbox{\boldmath $\sigma$}}^\varepsilon$ satisfies
\begin{equation}\label{lemma 3.2}
\| L_{\mbox{\boldmath $\lambda$},\mbox{\boldmath $\sigma$}}^\varepsilon(\Phi)\|_{\mathcal{C}_\varepsilon}\geq C\|\Phi\|_{\mathcal{C}_\varepsilon}\ \ \ \ \forall\ \Phi\in\mathcal{K}_{\mbox{\boldmath $\lambda$},\mbox{\boldmath $\sigma$}}^\varepsilon.
\end{equation}
\end{lemma}
\begin{proof} We omit it since it is similarly to Lemma 5.1 in \cite{Chio12}.
\end{proof}

\begin{proposition} \label{proposition-invertibility}
The inverse $(L_{\mbox{\boldmath $\lambda$},\mbox{\boldmath $\sigma$}}^\varepsilon)^{-1}$ of $L_{\mbox{\boldmath $\lambda$},\mbox{\boldmath $\sigma$}}^\varepsilon:\ \mathcal{K}_{\mbox{\boldmath $\lambda$},\mbox{\boldmath $\sigma$}}^\varepsilon\rightarrow\mathcal{K}_{\mbox{\boldmath $\lambda$},\mbox{\boldmath $\sigma$}}^\varepsilon$ exists for any $\varepsilon>0$ small and $(\mbox{\boldmath $\lambda$},\mbox{\boldmath $\sigma$})\in\mathcal{O}_\eta$. Besides, if $\varepsilon$ is small enough, its operator norm is uniformly bounded in $\varepsilon$ and $(\mbox{\boldmath $\lambda$},\mbox{\boldmath $\sigma$})\in\mathcal{O}_\eta$.
\end{proposition}
\begin{proof}
The proof is similarly to Proposition 5.2 in \cite{Chio12} and thus is omitted here.
\end{proof}

\begin{proposition}\label{proposition-reducement-1}
For any sufficiently small $\eta>0$, there exist $\varepsilon_0>0$ and a constant $C>0$ such that for any $\varepsilon\in(0,\varepsilon_0)$ and any $(\mbox{\boldmath $\lambda$},\mbox{\boldmath $\sigma$})\in\mathcal{O}_\eta$, there exists a unique solution $\Phi_{\mbox{\boldmath $\lambda$},\mbox{\boldmath $\sigma$}}^\varepsilon\in\mathcal{K}_{\mbox{\boldmath $\lambda$},\mbox{\boldmath $\sigma$}}^\varepsilon$ satisfying
\begin{equation}\label{proposition 3.4.1}
\Pi_{\mbox{\boldmath $\lambda$},\mbox{\boldmath $\sigma$}}^\varepsilon\left\{\sum_{i=1}^ka_i\mathcal{P}_\varepsilon W_i+\Phi_{\mbox{\boldmath $\lambda$},\mbox{\boldmath $\sigma$}}^\varepsilon-i_\varepsilon^*\left[f_\varepsilon\left(\sum_{i=1}^ka_i\mathcal{P}_\varepsilon w_i+i_\varepsilon(\Phi_{\mbox{\boldmath $\lambda$},\mbox{\boldmath $\sigma$}}^\varepsilon)\right)\right]\right\}=0,
\end{equation}
and
\begin{equation}\label{proposition 3.4.2}
\|\Phi_{\mbox{\boldmath $\lambda$},\mbox{\boldmath $\sigma$}}^\varepsilon\|_{\mathcal{C}_\varepsilon}\leq
\left\{\aligned
&C\varepsilon^{\frac{N+2s}{2}\alpha_0}\ \ \ \ \ \ \ \ \ \ \mbox{if}\ \ N>6s,\\
&C({\varepsilon+\varepsilon|\ln\varepsilon|})\ \ \ \ \mbox{if}\ \ N=6s,\\
&C\varepsilon\ \ \ \ \ \ \ \ \ \ \ \ \ \ \ \ \ \ \mbox{if}\ \ 2s<N<6s,
\endaligned \right.
\end{equation}
where $\alpha_0=\frac{1}{N-2s}$. Furthermore, the map $\Phi_{\mbox{\boldmath $\lambda$},\mbox{\boldmath $\sigma$}}^\varepsilon:\ \ \mathcal{O}_\eta\rightarrow\mathcal{K}_{\mbox{\boldmath $\lambda$},\mbox{\boldmath $\sigma$}}^\varepsilon$ is $C^1(\mathcal{O}_\eta)$.
\end{proposition}
\begin{proof}
First of all we point out that $\Phi_{\mbox{\boldmath $\lambda$},\mbox{\boldmath $\sigma$}}^\varepsilon$ is a solution of equation (\ref{proposition 3.4.1}) if and only if $\Phi_{\mbox{\boldmath $\lambda$},\mbox{\boldmath $\sigma$}}^\varepsilon$  is a fixed point of operator $T_{\mbox{\boldmath $\lambda$},\mbox{\boldmath $\sigma$}}^\varepsilon:\ \mathcal{K}_{\mbox{\boldmath $\lambda$},\mbox{\boldmath $\sigma$}}^\varepsilon\rightarrow\mathcal{K}_{\mbox{\boldmath $\lambda$},\mbox{\boldmath $\sigma$}}^\varepsilon$ defined by
\begin{align*}
&T_{\mbox{\boldmath $\lambda$},\mbox{\boldmath $\sigma$}}^\varepsilon(\Phi)=
(L_{\mbox{\boldmath $\lambda$},\mbox{\boldmath $\sigma$}}^\varepsilon)^{-1}N_\varepsilon(\Phi)\ \ \ \ \mbox{for}\ \ \Phi\in\mathcal{K}_{\mbox{\boldmath $\lambda$},\mbox{\boldmath $\sigma$}}^\varepsilon,
\end{align*}
where
\begin{align*}
&N_\varepsilon(\Phi)=
\Pi_{\mbox{\boldmath $\lambda$},\mbox{\boldmath $\sigma$}}^\varepsilon i_\varepsilon^*\left[f_\varepsilon\left(\sum_{i=1}^ka_i\mathcal{P}_\varepsilon w_i+i_\varepsilon(\Phi)\right)-\sum_{i=1}^ka_if_0(w_i)-f'_0\left(\sum_{i=1}^ka_i\mathcal{P}_\varepsilon w_i\right)i_\varepsilon(\Phi)\right].
\end{align*}
The claim will follow by showing that $T_{\mbox{\boldmath $\lambda$},\mbox{\boldmath $\sigma$}}^\varepsilon$ is a contraction mapping on $\mathcal{K}_{\mbox{\boldmath $\lambda$},\mbox{\boldmath $\sigma$}}^\varepsilon:=\{\Phi\in\mathcal{K}_{\mbox{\boldmath $\lambda$},\mbox{\boldmath $\sigma$}}^\varepsilon:\ \Phi\ \mbox{satisfies}\ (\ref{proposition 3.4.2})\}$.
By Lemma \ref{lemma-estimate of error-1}, Lemma \ref{(M17)-Lemma A} and (\ref{lemma A.7-1}) in Lemma \ref{lemma A.7}, we get
\begin{align}\label{estimation-1}
&\|T_{\mbox{\boldmath $\lambda$},\mbox{\boldmath $\sigma$}}^\varepsilon(\Phi)\|_{\mathcal{C}_\varepsilon}\leq
C\left\|f_\varepsilon\left(\sum_{i=1}^ka_i\mathcal{P}_\varepsilon w_i+i_\varepsilon(\Phi)\right)-\sum_{i=1}^ka_if_0(w_i)-f'_0\left(\sum_{i=1}^ka_i\mathcal{P}_\varepsilon w_i\right)i_\varepsilon(\Phi)\right\|_{L^\frac{2N}{N+2s}(\Omega_\varepsilon)}\nonumber\\
&\ \ \ \ \ \ \ \ \ \ \ \ \ \ \ \ \ \leq C\Bigg\|f_\varepsilon\left(\sum_{i=1}^ka_i\mathcal{P}_\varepsilon w_i+i_\varepsilon(\Phi)\right)-f_\varepsilon\left(\sum_{i=1}^ka_i\mathcal{P}_\varepsilon w_i\right)\nonumber\\
&\ \ \ \ \ \ \ \ \ \ \ \ \ \ \ \ \ \ \ \ -f'_\varepsilon\left(\sum_{i=1}^ka_i\mathcal{P}_\varepsilon w_i\right)i_\varepsilon(\Phi)\Bigg\|_{L^\frac{2N}{N+2s}(\Omega_\varepsilon)}\nonumber\\
&\ \ \ \ \ \ \ \ \ \ \ \quad\quad\quad +C\left\|\left[f'_\varepsilon\left(\sum_{i=1}^ka_i\mathcal{P}_\varepsilon w_i\right)-f'_0\left(\sum_{i=1}^ka_i\mathcal{P}_\varepsilon w_i\right)\right]i_\varepsilon(\Phi)\right\|_{L^\frac{2N}{N+2s}(\Omega_\varepsilon)}\nonumber\\
&\ \ \ \ \ \ \ \ \ \ \ \quad\quad\quad +C\left\| f_\varepsilon\left(\sum_{i=1}^ka_i\mathcal{P}_\varepsilon w_i\right)-f_0\left(\sum_{i=1}^ka_i\mathcal{P}_\varepsilon w_i\right)\right\|_{L^\frac{2N}{N+2s}(\Omega_\varepsilon)}\nonumber\\
&\ \ \ \ \ \ \ \ \ \ \ \quad\quad\quad +C\left\| f_0\left(\sum_{i=1}^ka_i\mathcal{P}_\varepsilon w_i\right)-\sum_{i=1}^ka_if_0 (w_i)\right\|_{L^\frac{2N}{N+2s}(\Omega_\varepsilon)}.
\end{align}
It is easy to see that
\begin{align}\label{estimation-2}
&\left\|f_\varepsilon\left(\sum_{i=1}^ka_i\mathcal{P}_\varepsilon w_i+i_\varepsilon(\Phi)\right)-f_\varepsilon\left(\sum_{i=1}^ka_i\mathcal{P}_\varepsilon w_i\right)-f'_\varepsilon\left(\sum_{i=1}^ka_i\mathcal{P}_\varepsilon w_i\right)i_\varepsilon(\Phi)\right\|_{L^\frac{2N}{N+2s}(\Omega_\varepsilon)}\nonumber\\
&\leq C\| i_\varepsilon(\Phi)\|_{L^\frac{2N}{N+2s}(\Omega_\varepsilon)}^{min\{2,p\}}\nonumber\\
&\leq C\|\Phi\|_{\mathcal{C}_\varepsilon}^{min\{2,p\}}
\end{align}
and by (\ref{lemma A.7-2}) of Lemma {\ref{lemma A.7}} that
\begin{align}\label{estimation-3}
&\left\|\left[f'_\varepsilon\left(\sum_{i=1}^ka_i\mathcal{P}_\varepsilon w_i\right)-f'_0\left(\sum_{i=1}^ka_i\mathcal{P}_\varepsilon w_i\right)\right]i_\varepsilon(\Phi)\right\|_{L^\frac{2N}{N+2s}(\Omega_\varepsilon)}\nonumber\\
&\leq\left\| f'_\varepsilon\left(\sum_{i=1}^ka_i\mathcal{P}_\varepsilon w_i\right)-f'_0\left(\sum_{i=1}^ka_i\mathcal{P}_\varepsilon w_i\right)\right\|_{L^\frac{N}{2s}(\Omega_\varepsilon)}\| i_\varepsilon(\Phi)\|_{L^\frac{2N}{N-2s}(\Omega_\varepsilon)}\nonumber\\
&\leq C\varepsilon|\ln\varepsilon|\|\Phi\|_{\mathcal{C}_\varepsilon}.
\end{align}
By using Lemma \ref{(M17)-Lemma A}, (\ref{estimation-1}), (\ref{estimation-2}) and (\ref{estimation-3}), we deduce that if $\Phi_{\mbox{\boldmath $\lambda$},\mbox{\boldmath $\sigma$}}^\varepsilon$ satisfies (\ref{proposition 3.4.2}), that is , $\Phi_{\mbox{\boldmath $\lambda$},\mbox{\boldmath $\sigma$}}^\varepsilon\leq C_1(\varepsilon+\gamma(\varepsilon))$, then there exists $C_1>0$ such that $\| T_{\mbox{\boldmath $\lambda$},\mbox{\boldmath $\sigma$}}^\varepsilon(\Phi)\|_{\mathcal{C}_\varepsilon}\leq C_1(\varepsilon+\gamma(\varepsilon))$.
Arguing as in the previous step, we can prove that if $\Phi_1$ and $\Phi_2$ satisfy (\ref{proposition 3.4.2}) then
\begin{align*}
&\|T_{\mbox{\boldmath $\lambda$},\mbox{\boldmath $\sigma$}}^\varepsilon(\Phi_1)-T_{\mbox{\boldmath $\lambda$},\mbox{\boldmath $\sigma$}}^\varepsilon(\Phi_2)\|_{\mathcal{C}_\varepsilon}\nonumber\\
&=\|(L_{\mbox{\boldmath $\lambda$},\mbox{\boldmath $\sigma$}}^\varepsilon)^{-1}(N_\varepsilon(\Phi_1))-(L_{\mbox{\boldmath $\lambda$},\mbox{\boldmath $\sigma$}}^\varepsilon)^{-1}(N_\varepsilon(\Phi_2))\|_{\mathcal{C}_\varepsilon}\nonumber\\
&=\Bigg\|(L_{\mbox{\boldmath $\lambda$},\mbox{\boldmath $\sigma$}}^\varepsilon)^{-1}\Pi_{\mbox{\boldmath $\lambda$},\mbox{\boldmath $\sigma$}}^\varepsilon i_\varepsilon^*\Bigg[f_\varepsilon\left(\sum_{i=1}^ka_i\mathcal{P}_\varepsilon w_i+i_\varepsilon(\Phi_1)\right)
-f_\varepsilon\left(\sum_{i=1}^ka_i\mathcal{P}_\varepsilon w_i+i_\varepsilon(\Phi_2)\right)\nonumber\\
&\ \ \ -f'_0\left(\sum_{i=1}^ka_i\mathcal{P}_\varepsilon w_i\right)(i_\varepsilon(\Phi_1)-i_\varepsilon(\Phi_2))\Bigg]\Bigg\|_{\mathcal{C}_\varepsilon}\nonumber\\
&\leq C_2\Bigg\|f_\varepsilon\left(\sum_{i=1}^ka_i\mathcal{P}_\varepsilon w_i+i_\varepsilon(\Phi_1)\right)
-f_\varepsilon\left(\sum_{i=1}^ka_i\mathcal{P}_\varepsilon w_i+i_\varepsilon(\Phi_2)\right)\nonumber\\
&\ \quad-f'_\varepsilon\left(\sum_{i=1}^ka_i\mathcal{P}_\varepsilon w_i+i_\varepsilon(\Phi_2)\right)(i_\varepsilon(\Phi_1)-i_\varepsilon(\Phi_2))\Bigg\|_{L^\frac{2N}{N+2s}(\Omega_\varepsilon)}\nonumber\\
&\ \quad+C_2\left\| \left[f'_\varepsilon\left(\sum_{i=1}^ka_i\mathcal{P}_\varepsilon w_i+i_\varepsilon(\Phi_2)\right)-f'_\varepsilon\left(\sum_{i=1}^ka_i\mathcal{P}_\varepsilon w_i\right)\right](i_\varepsilon(\Phi_1)-i_\varepsilon(\Phi_2))\right\|_{L^\frac{2N}{N+2s}(\Omega_\varepsilon)}\nonumber\\
&\ \quad+C_2\left\| \left[f'_\varepsilon\left(\sum_{i=1}^ka_i\mathcal{P}_\varepsilon w_i\right)-f'_0\left(\sum_{i=1}^ka_i\mathcal{P}_\varepsilon w_i\right)\right](i_\varepsilon(\Phi_1)-i_\varepsilon(\Phi_2))\right\|_{L^\frac{2N}{N+2s}(\Omega_\varepsilon)}\nonumber\\
&\leq L\|\Phi_1-\Phi_2\|_{\mathcal{C}_\varepsilon}.
\end{align*}
for some $L\in(0,1)$. The remaining parts are obtained by standard arguments, see \cite{Musso31}.
\end{proof}

It is easy to know that for any fixed $\varepsilon>0$, $V\in H_{0,L}^s(\Omega_\varepsilon)$ is a weak solution to (\ref{dilated equation}) if and only if it is a critical point of the energy functional $E_\varepsilon: H_{0,L}^s(\mathcal{C}_\varepsilon)\rightarrow\mathbb R$ defined by
\begin{equation}\label{energy functional}
E_\varepsilon(V)=\frac{1}{2k_s}\int_{\mathcal{C}_\varepsilon} t^{1-2s}|\nabla V|^2dxdt-\int_{\Omega_\varepsilon\times\{0\}}F_\varepsilon(i_\varepsilon(V))dx,
\end{equation}
where $F_\varepsilon(t)=\int_0^tf_\varepsilon(t)dt$. Notice that $E_\varepsilon(V)$ is a $C^1$-functional and
\begin{equation}\label{C1-functional}
E'_\varepsilon(V)\Phi=\frac{1}{k_s}\int_{\mathcal{C}_\varepsilon} t^{1-2s}\nabla V\cdot\nabla \Phi dxdt-\int_{\Omega_\varepsilon\times\{0\}}f_\varepsilon(i_\varepsilon(V))i_\varepsilon(\Phi)dx\ \ \ \mbox{for any}\ \ \Phi\in H_{0,L}^s(\mathcal{C}_\varepsilon).
\end{equation}
Now we introduce the function
\begin{equation}\label{func}
I_\varepsilon({\mbox{\boldmath $\lambda$},\mbox{\boldmath $\sigma$}})=E_\varepsilon\left(\sum_{i=1}^ka_i\mathcal{P}_\varepsilon W_i+\Phi_{\mbox{\boldmath $\lambda$},\mbox{\boldmath $\sigma$}}^\varepsilon\right)
\end{equation}
for $(\mbox{\boldmath $\lambda$},\mbox{\boldmath $\sigma$})=((\lambda_1,...,\lambda_k), (\sigma_1,...,\sigma_k))\in\mathcal{O}_\eta$.

Let $\alpha_0=\frac{1}{N-2s}$ here and in the sequel. Arguing as Proposition 5.4 in \cite{Chio12} and Lemma 2.6 in \cite{Rois15}, we can obtain the following result.
\begin{proposition}\label{proposition-expansion of L-1}
(1) Suppose $\varepsilon>0$ is sufficiently small. If $(\mbox{\boldmath $\lambda$},\mbox{\boldmath $\sigma$})$ is a critical point of the function $I_\varepsilon({\mbox{\boldmath $\lambda$},\mbox{\boldmath $\sigma$}})$, then the function $V=\sum\limits_{i=1}^ka_i\mathcal{P}_\varepsilon W_i+\Phi_{\mbox{\boldmath $\lambda$},\mbox{\boldmath $\sigma$}}^\varepsilon$ is a solution to (\ref{dilated equation}). Hence the changing variables $U(z)=\varepsilon^{-\frac{1}{2-\frac{\varepsilon(N-2s)}{2s}}}V_\varepsilon(\varepsilon^{-\frac{1}{N-2s}}z)$ is the solution of (\ref{equi}) for $z\in\mathcal{C}$.

(2) For $\varepsilon\rightarrow0$, there holds
\begin{equation}\label{equa-1}
I_\varepsilon({\mbox{\boldmath $\lambda$},\mbox{\boldmath $\sigma$}})=\frac{ksc_0}{N}-\frac{\varepsilon kc_0}{(p+1)^2}+\frac{1}{2}\varepsilon\Upsilon_k(\mbox{\boldmath $\lambda$},\mbox{\boldmath $\sigma$})+\frac{\varepsilon k}{p+1}\int_{\mathbb R^N}w^{p+1}\log wdx+o(\varepsilon)
\end{equation}
in $C^1$-uniformly with respect to $({\mbox{\boldmath $\lambda$},\mbox{\boldmath $\sigma$}})\in\mathcal{O}_\eta$. Here
\begin{align}\label{equa-2}
&\Upsilon_k(\mbox{\boldmath $\lambda$},\mbox{\boldmath $\sigma$})=c_1^2\left(\sum_{i=1}^k\lambda_i^{N-2s}H(\sigma_i,\sigma_i)-\sum\limits_{i,h=1,i\neq h}^ka_ia_hG(\sigma_i,\sigma_h)(\lambda_i\lambda_h)^\frac{N-2s}{2}\right)\nonumber\\
&\ \ \ \ \ \ \ \ \ \quad\quad-\frac{c_0(N-2s)}{p+1}\log(\lambda_1\cdot\cdot\cdot\cdot\lambda_k),
\end{align}
\begin{equation}\label{equa-3}
c_0=\int_{\mathbb R^N}w^{p+1}dx
\end{equation}
and
\begin{equation}\label{equa-4}
c_1=\int_{\mathbb R^N}w^pdx,
\end{equation}
where $w:=w_{1,0}$. 
\end{proposition}

\begin{proof}
 We first prove (1). Setting $\overline{V}=\sum_{i=1}^ka_i\mathcal{P}_\varepsilon W_i$ for the sake of simplicity. Applying $I'_\varepsilon({\mbox{\boldmath $\lambda$},\mbox{\boldmath $\sigma$}})=0$, (\ref{proposition 3.4.1}) and (\ref{proposition 3.4.2}), we get
\begin{align*}
&\frac{\partial I_\varepsilon }{\partial\varrho}=E'_\varepsilon(\overline{V}+\Phi_{\mbox{\boldmath $\lambda$},\mbox{\boldmath $\sigma$}}^\varepsilon)\cdot\left(\frac{\partial\overline{V}}{\partial\varrho}+\frac{\partial\Phi_{\mbox{\boldmath $\lambda$},\mbox{\boldmath $\sigma$}}^\varepsilon}{\partial\varrho}\right)\nonumber\\
&\ \ \quad=\left(\overline{V}+\Phi_{\mbox{\boldmath $\lambda$},\mbox{\boldmath $\sigma$}}^\varepsilon-i_{\varepsilon}^*f_\varepsilon(\overline{V}+i_{\varepsilon}(\Phi_{\mbox{\boldmath $\lambda$},\mbox{\boldmath $\sigma$}})), \frac{\partial\overline{V}}{\partial\varrho}+\frac{\partial\Phi_{\mbox{\boldmath $\lambda$},\mbox{\boldmath $\sigma$}}^\varepsilon}{\partial\varrho}\right)_\mathcal{C_\varepsilon}\nonumber\\
&\ \ \quad=\sum_{h=1}^k\sum_{l=0}^Nc_{hl}\left( \mathcal{P}_\varepsilon\Psi_h^l, \frac{\partial\overline{V}}{\partial\varrho}+\frac{\partial\Phi_{\mbox{\boldmath $\lambda$},\mbox{\boldmath $\sigma$}}^\varepsilon}{\partial\varrho}\right)_\mathcal{C_\varepsilon}\nonumber\\
&\ \ \quad=\sum_{h=1}^k\sum_{l=0}^Nc_{hl}\left[\left(\mathcal{P}_\varepsilon\Psi_h^l, \sum_{i=1}^ka_i\mathcal{P}_\varepsilon
\frac{\partial W_i}{\partial\varrho}\right)_\mathcal{C_\varepsilon}-\left(\mathcal{P}_\varepsilon
\frac{\partial \Psi_h^l}{\partial\varrho},\Phi_{\mbox{\boldmath $\lambda$},\mbox{\boldmath $\sigma$}}^\varepsilon\right)_\mathcal{C_\varepsilon}\right]\nonumber\\
&\ \ \quad=0,
\end{align*}
where $\varrho$ is one of $\lambda_i$ and $\sigma_i^j$ with $ i=1,2,...,k$ and $j=1,...,N$ , $c_{hl}\in\mathbb R$. We also can conclude that $c_{hl}=0$ for all $h$ and $l$, which implies that the function $V$ is a solution of the equation (\ref{dilated equation}), and hence $U(x)$ is a solution to (\ref{equi}) for $\varepsilon>0$ sufficiently small.

Now we give the proof of (2). Using (\ref{proposition 3.4.2}), we can obtain that
\begin{align*}
&I_\varepsilon({\mbox{\boldmath $\lambda$},\mbox{\boldmath $\sigma$}})=E_\varepsilon\left(\sum_{i=1}^ka_i\mathcal{P}_\varepsilon W_i+\Phi_{\mbox{\boldmath $\lambda$},\mbox{\boldmath $\sigma$}}^\varepsilon\right)=E_\varepsilon\left(\sum_{i=1}^ka_i\mathcal{P}_\varepsilon W_i\right)+o(\varepsilon)\nonumber\\
&\ \ \ \ \ \ \ \ \quad=\frac{1}{2k_s}\int_{\mathcal{C}_\varepsilon} t^{1-2s}\left|\nabla \left(\sum_{i=1}^ka_i\mathcal{P}_\varepsilon W_i\right)\right|^2dxdt-\frac{1}{p+1-\varepsilon}\int_{\Omega_\varepsilon\times\{0\}}\left|\sum_{i=1}^ka_i\mathcal{P}_\varepsilon W_i\right|^{p+1-\varepsilon}dx\nonumber\\
&\ \ \ \ \ \ \ \ \ \ \ \quad+o(\varepsilon).
\end{align*}
We decompose
\begin{equation}\label{decomp}
E_\varepsilon\left(\sum_{i=1}^ka_i\mathcal{P}_\varepsilon W_i\right)=E_0\left(\sum_{i=1}^ka_i\mathcal{P}_\varepsilon W_i\right)+\left[E_\varepsilon\left(\sum_{i=1}^ka_i\mathcal{P}_\varepsilon W_i\right)-E_0\left(\sum_{i=1}^ka_i\mathcal{P}_\varepsilon W_i\right)\right],
\end{equation}
so it suffice to estimate the above two terms. It is easy to see that
\begin{align}\label{terms}
&E_0\left(\sum_{i=1}^ka_i\mathcal{P}_\varepsilon W_i\right)
=\frac{1}{2k_s}\int_{\mathcal{C}_\varepsilon} t^{1-2s}\left|\nabla\left(\sum_{i=1}^ka_i\mathcal{P}_\varepsilon W_i\right)\right|^2dxdt-\frac{1}{p+1}\int_{\Omega_\varepsilon\times\{0\}}\left|\sum_{i=1}^ka_i\mathcal{P}_\varepsilon w_i\right|^{p+1}dx.
\end{align}

Setting $B_i=B_N(\sigma_i,\eta/2)$, where $\eta$ is defined in (\ref{admissible set}), applying Lemma \ref{lemma A.1} and Lemma \ref{lemma A.2}, we can deduce that
\begin{align*}
&\int_{\Omega_\varepsilon}w_i^p\mathcal{P}_\varepsilon w_idx=\int_{\Omega_\varepsilon}w_i^{p+1}dx+\int_{\Omega_\varepsilon}w_i^p(\mathcal{P}_\varepsilon w_i-w_i)dx\nonumber\\
&\ \ \ \ \ \ \ \quad\quad\quad\quad=c_0-\varepsilon c_1\lambda_i^\frac{N-2s}{2}\int_{\Omega_\varepsilon}w_i^pH(\varepsilon^{\alpha_0}x,\sigma_i)dx+o(\varepsilon)\nonumber\\
&\ \ \ \ \ \ \ \quad\quad\quad\quad=c_0-\varepsilon c_1^2\lambda_i^{N-2s}H(\sigma_i,\sigma_i)+o(\varepsilon),\nonumber\\
&\int_{\Omega_\varepsilon}w_h^p\mathcal{P}_\varepsilon w_idx=\int_\frac{B_i}{\varepsilon^{\alpha_0}}w_h^p\mathcal{P}_\varepsilon w_idx+o(\varepsilon)\nonumber\\
&\ \ \ \ \ \ \ \quad\quad\quad\quad=\int_\frac{B_i}{\varepsilon^{\alpha_0}}\varepsilon c_1\lambda_i^\frac{N-2s}{2}w_h^pG(\varepsilon^{\alpha_0}x,\sigma_i)dx+o(\varepsilon)\nonumber\\
&\ \ \ \ \ \ \ \quad\quad\quad\quad=\varepsilon c_1^2(\lambda_i\lambda_h)^\frac{N-2s}{2}G(\sigma_i,\sigma_h)+o(\varepsilon),\nonumber\\
\end{align*}
for $i,h=1,2,...,k$ and $i\neq h$, where $G$ and $H$ are the functions defined in (\ref{Green}) and (\ref{Green-regular}), $c_0$ and $c_1$ are defined in (\ref{equa-3}) and (\ref{equa-4}), respectively. Integrating by parts and then the estimates obtained above yield that 
\begin{align}\label{estim-1}
&\frac{1}{2k_s}\int_{\mathcal{C}_\varepsilon} t^{1-2s}\left|\nabla \left(\sum_{i=1}^ka_i\mathcal{P}_\varepsilon W_i\right)\right|^2dxdt\nonumber\\
&=\frac{1}{2}\sum_{i=1}^k\int_{\Omega_\varepsilon}w_i^p\mathcal{P}_\varepsilon w_idx+\frac{1}{2}
\sum\limits_{i,h=1,i\neq h}^k a_ia_h\int_{\Omega_\varepsilon}w_h^p\mathcal{P}_\varepsilon w_idx\nonumber\\
&=\frac{1}{2}\sum_{i=1}^k[c_0-\varepsilon c_1^2\lambda_i^{N-2s}H(\sigma_i,\sigma_i)+o(\varepsilon)]+\frac{1}{2}\sum\limits_{i,h=1,i\neq h}^k\varepsilon a_ia_h
c_1^2(\lambda_i\lambda_h)^\frac{N-2s}{2}G(\sigma_i,\sigma_h)+o(\varepsilon)\nonumber\\
&=\frac{kc_0}{2}-\frac{c_1^2\varepsilon}{2}\left[\sum_{i=1}^k\lambda_i^{N-2s}H(\sigma_i,\sigma_i)-\sum\limits_{i,h=1,i\neq h}^ka_ia_h
(\lambda_i\lambda_h)^\frac{N-2s}{2}G(\sigma_i,\sigma_h)\right]+o(\varepsilon).
\end{align}
On the other hand, we see that 
\begin{align}\label{estim-2}
&\int_{\Omega_\varepsilon}\left|a_i\mathcal{P}_\varepsilon w_i\right|^{p+1}dx=\int_{\Omega_\varepsilon}\left|\mathcal{P}_\varepsilon w_i\right|^{p+1}dx
=c_0-\varepsilon (p+1)c_1^2\lambda_i^{N-2s}H(\sigma_i,\sigma_i)+o(\varepsilon),
\end{align}
\begin{align}\label{estim-3}
&\int_{\Omega_\varepsilon}\left(\left|\sum_{i=1}^ka_i\mathcal{P}_\varepsilon w_i\right|^{p+1}-\sum_{i=1}^k\left|a_i\mathcal{P}_\varepsilon w_i\right|^{p+1}\right)dx\nonumber\\
&=\varepsilon c_1^2(p+1)\sum\limits_{i,h=1,i\neq h}^ka_ia_h
(\lambda_i\lambda_h)^\frac{N-2s}{2}G(\sigma_i,\sigma_h)+o(\varepsilon).
\end{align}
From the estimates obtained in the previous paragraph, we can conclude that 
\begin{align}\label{estim-4}
&E_0\left(\sum_{i=1}^ka_i\mathcal{P}_\varepsilon W_i\right)=\frac{ksc_0}{N}+\frac{1}{2}\varepsilon c_1^2\sum_{i=1}^k\lambda_i^{N-2s}H(\sigma_i,\sigma_i)\nonumber\\
&\ \ \ \ \ \ \ \ \ \ \ \ \ \ \ \ \ \ \ \ \ \ \ \ \ \ -\frac{1}{2}\varepsilon c_1^2\sum\limits_{i,h=1,i\neq h}^ka_ia_hG(\sigma_i,\sigma_h)(\lambda_i\lambda_h)^\frac{N-2s}{2}+o(\varepsilon).
\end{align}

As we have seen, it has 
\begin{align}\label{estim-5}
&\int_{\Omega_\varepsilon}\left|\sum_{i=1}^ka_i\mathcal{P}_\varepsilon w_i\right|^{p+1}dx=kc_0+o(1),
\end{align}
\begin{align}\label{estim-6}
&\int_{\Omega_\varepsilon}\left|\sum_{i=1}^ka_i\mathcal{P}_\varepsilon W_i\right|^{p+1}\log\left|\sum_{i=1}^ka_i\mathcal{P}_\varepsilon W_i\right|dx\nonumber\\
&=-\frac{c_0(N-2s)}{2}\log(\lambda_1\cdot\cdot\cdot\cdot\lambda_k)+k\int_{\mathbb R^N}w^{p+1}\log wdx+o(1).
\end{align}
The second equality (\ref{estim-6}) can be computed as Lemma 2.6 in \cite{Rois15}, Lemma 6.2 \cite{Musso34} and \cite{Del33}.
Moreover, by using Taylor's expansion, (\ref{estim-5}) and (\ref{estim-6}), we can conclude that 
\begin{align}\label{estim-7}
&E_\varepsilon\left(\sum_{i=1}^ka_i\mathcal{P}_\varepsilon W_i\right)-E_0\left(\sum_{i=1}^ka_i\mathcal{P}_\varepsilon W_i\right)\nonumber\\
&=\frac{1}{p+1}\int_{\Omega_\varepsilon}\left|\sum_{i=1}^ka_i\mathcal{P}_\varepsilon W_i\right|^{p+1}dx-\frac{1}{p+1-\varepsilon}\int_{\Omega_\varepsilon}\left|\sum_{i=1}^ka_i\mathcal{P}_\varepsilon W_i\right|^{p+1-\varepsilon}dx\nonumber\\
&=-{\frac{\varepsilon}{(p+1)^2}}\int_{\Omega_\varepsilon}\left|\sum_{i=1}^ka_i\mathcal{P}_\varepsilon W_i\right|^{p+1}dx+\frac{\varepsilon}{p+1}\int_{\Omega_\varepsilon}\left|\sum_{i=1}^ka_i\mathcal{P}_\varepsilon W_i\right|^{p+1}\log\left|\sum_{i=1}^ka_i\mathcal{P}_\varepsilon W_i\right|dx+o(\varepsilon)\nonumber\\
&=-\frac{\varepsilon kc_0}{(p+1)^2}+\frac{\varepsilon k}{p+1}\int_{\mathbb R^N}w^{p+1}\log wdx
-\frac{c_0\varepsilon(N-2s)}{2(p+1)}\log(\lambda_1\cdot\cdot\cdot\cdot\lambda_k)+o(\varepsilon).
\end{align}
Then by (\ref{terms}) and (\ref{estim-7}), the proof is complete.
\end{proof}

Now we consider the case $k=2$ and suppose $a_1=1$ and $a_2=-1$. We introduce the set
\begin{equation}\label{set}
\Lambda:=\{(\mbox{\boldmath $\lambda$},\mbox{\boldmath $\sigma$})=(\lambda_1,\ \lambda_2,\ \sigma_1,\ \sigma_2): \lambda_1>0,\ \lambda_2>0,\ \sigma_1\in\Omega,\ \sigma_2\in\Omega\ \mbox{and}\ \sigma_1\neq\sigma_2\}
\end{equation}
and the function $\Upsilon_2:\Lambda\rightarrow \mathbb R$ defined by
\begin{align}\label{define}
&\Upsilon_2(\mbox{\boldmath $\lambda$},\mbox{\boldmath $\sigma$})=c_1^2[H(\sigma_1,\sigma_1)\lambda_1^{N-2s}+H(\sigma_2,\sigma_2)\lambda_2^{N-2s}+2G(\sigma_1,\sigma_2)\lambda_1^\frac{N-2s}{2}\lambda_2^\frac{N-2s}{2}]\nonumber\\
&\ \ \ \ \ \ \ \ \ \quad\quad-\frac{c_0(N-2s)}{p+1}\log(\lambda_1\lambda_2).
\end{align}

\begin{lemma}\label{lemma-expansion of L-1}
If $(\lambda^*,\sigma^*)$ is a critical point of $\Upsilon_2$, then $\sigma^*$ is a critical point of $\varphi$. If $(\lambda^*,\sigma^*)$ is a minimal point of $\Upsilon_2$, then $\sigma^*$ is a minimal point of $\varphi$.
\end{lemma}
\begin{proof}
The proof is similarly to Lemma 3.2 in \cite{Bartsch1} and thus is omitted here.
\end{proof}

\noindent\textbf{Proof of Theorem 1.1.}
Similarly as Theorem 1.1 of \cite{Bartsch1}, the above lemmas and propositions give the result.
  \hfill$\Box$

\appendix
\section{\textbf{Appendix}\label{Appendix A}}

In this section, we collect some technical lemmas from \cite{Chio12} and give some basic estimations needed.

By using the definition of $w_{\lambda,\xi}$, $\psi_{\lambda,\xi}^j$, $\mathcal{P}_\varepsilon w_{\lambda,\xi}$ and $\mathcal{P}_\varepsilon \psi_{\lambda,\delta}^j$ $(\mbox{for}\ \ i=1,...,k\ \mbox{and}\ \ j=1,...,N)$, we get
\begin{align}\label{Appendix A.1}
&\psi_{\lambda,\sigma}^0(x)=\frac{\partial w_{\lambda,\sigma}}{\partial \lambda}(x)\nonumber\\
&\ \ \ \ \ \ \ \ \ \ =a_{N,s}\frac{(N-2s)}{2}\lambda^\frac{N-2s-2}{2}\frac{|x-\sigma|^2-\lambda^2}{(\lambda^2+|x-\sigma|^2)^\frac{N-2s+2}{2}}\ \ \ \ x\in\mathbb R^N,
\end{align}
\begin{align}\label{Appendix A.2}
&\psi_{\lambda,\sigma}^j(x)=\frac{\partial w_{\lambda,\sigma}}{\partial \sigma_j}(x)\nonumber\\
&\ \ \ \ \ \ \ \ \ \ =-a_{N,s}(N-2s)\lambda^\frac{N-2s}{2}\frac{x_j-\sigma_j}{(\lambda^2+|x-\sigma|^2)^\frac{N-2s+2}{2}}\ \ \ \ \ x\in\mathbb R^N.
\end{align}
In particular it holds
\begin{align}\label{Appendix A.3}
&\mathcal{P}_\varepsilon w_{\varepsilon^{\alpha_0}\lambda,\sigma}(x)=\varepsilon^{-\frac{(N-2s)\alpha_0}{2}}\mathcal{P}_\varepsilon w_{\lambda,\sigma\varepsilon^{-\alpha_0}}\left(\frac{x}{\varepsilon^{\alpha_0}}\right)\ \ \ \ \ x\in \Omega,
\end{align}
\begin{align}\label{Appendix A.4}
&\mathcal{P}_\varepsilon \psi_{\varepsilon^{\alpha_0}\lambda,\sigma}^j(x)=\varepsilon^{-\frac{(N-2s+2)\alpha_0}{2}}\mathcal{P}_\varepsilon \psi_{\lambda,\sigma\varepsilon^{-\alpha_0}}^j\left(\frac{x}{\varepsilon^{\alpha_0}}\right)\ \ \ \ \ x\in \Omega.
\end{align}

The first four lemmas are from Lemma C.1-C.4 in \cite{Chio12}.
\begin{lemma}\label{lemma A.1}
 Let $\lambda>0$ and $\sigma=(\sigma^1,...,\sigma^N)\in\Omega$. For any $x\in\Omega_\varepsilon$, there hold
\begin{eqnarray*}
&\mathcal{P}_\varepsilon w_{\lambda,\sigma\varepsilon^{-\alpha_0} }(x)=w_{\lambda,\sigma\varepsilon^{-\alpha_0}}(x)-c_1\lambda^\frac{N-2s}{2}H(\varepsilon^{\alpha_0}x,\sigma)\varepsilon^{(N-2s)\alpha_0}+o(\varepsilon^{(N-2s)\alpha_0}),
\end{eqnarray*}
\begin{eqnarray*}
&\ \ \ \ \ \ \ \mathcal{P}_\varepsilon \psi_{\lambda,\sigma\varepsilon^{-\alpha_0}}^j(x)=\psi_{\lambda,\sigma\varepsilon^{-\alpha_0}}^j(x)-c_1\lambda^\frac{N-2s}{2}\frac{\partial H}{\partial\sigma^j}(\varepsilon^{\alpha_0}x,\sigma)\varepsilon^{(N-2s+1)\alpha_0}+o(\varepsilon^{(N-2s+1)\alpha_0}),
\end{eqnarray*}
\begin{eqnarray*}
&\ \ \ \ \ \ \ \ \ \ \mathcal{P}_\varepsilon \psi_{\lambda,\sigma\varepsilon^{-\alpha_0}}^0(x)=\psi_{\lambda,\sigma\varepsilon^{-\alpha_0}}^0(x)-\frac{c_1(N-2s)}{2}\lambda^\frac{N-2s-2}{2} H(\varepsilon^{\alpha_0}x,\sigma)\varepsilon^{(N-2s)\alpha_0}+o(\varepsilon^{(N-2s)\alpha_0}),
\end{eqnarray*}
where $c_1$ is defined in (\ref{equa-4}). As $\varepsilon\rightarrow0$, $o\to 0$ uniformly in $x\in\Omega_\varepsilon$ and $\sigma\in\Omega$ provided $\mbox{dist}(\sigma,\partial\Omega)>\bar{C}$ for some constant $\bar{C}>0$.
\end{lemma}

\begin{lemma}\label{lemma A.2} Let $\lambda>0$ and $\sigma=(\sigma^1,...,\sigma^N)\in\Omega$, there hold
\begin{eqnarray*}\label{Appendix A.8}
&\mathcal{P}_\varepsilon w_{\lambda,\sigma\varepsilon^{-\alpha_0}}(x)=c_1\lambda^\frac{N-2s}{2}G(\varepsilon^{\alpha_0}x,\sigma)\varepsilon^{(N-2s)\alpha_0}+o(\varepsilon^{(N-2s)\alpha_0}),
\end{eqnarray*}
\begin{eqnarray*}\label{Appendix A.9}
&\ \ \ \ \ \ \ \ \mathcal{P}_\varepsilon \psi_{\lambda,\sigma\varepsilon^{-\alpha_0}}^j(x)=c_1\lambda^\frac{N-2s}{2}\frac{\partial G}{\partial\sigma^j}(\varepsilon^{\alpha_0}x,\sigma)\varepsilon^{(N-2s+1)\alpha_0}+o(\varepsilon^{(N-2s+1)\alpha_0}),
\end{eqnarray*}
\begin{eqnarray*}\label{Appendix A.10}
&\ \ \ \ \ \ \ \ \ \ \ \mathcal{P}_\varepsilon \psi_{\lambda,\sigma\varepsilon^{-\alpha_0}}^0(x)=\frac{c_1(N-2s)}{2}\lambda^\frac{N-2s-2}{2} G(\varepsilon^{\alpha_0}x,\sigma)\varepsilon^{(N-2s)\alpha_0}+o(\varepsilon^{(N-2s)\alpha_0}),
\end{eqnarray*}
where $c_1>0$ is the constant defined in (\ref{equa-4}). As $\varepsilon\rightarrow0$, $o\to 0$ uniformly in $x\in\Omega_\varepsilon$ and $\sigma\in\Omega$ provided $|\sigma-\varepsilon^{\alpha_0}x|>C$ and $\mbox{dist}(\partial\Omega,\varepsilon^{\alpha_0}x)>C$ for fixed $C>0$.
\end{lemma}

\begin{lemma}\label{lemma A.3}
For any $\varepsilon>0$, $i=1,...,k$ and $j=1,...,N$, there exists $C>0$ such that
\begin{eqnarray}\label{Appendix A.11}
&\|\mathcal{P}_\varepsilon w_i\|_{L^\frac{2N}{N-2s}(\Omega_\varepsilon)}\leq\|w_i\|_{L^\frac{2N}{N-2s}(\Omega_\varepsilon)}\leq C,
\end{eqnarray}
\begin{eqnarray}\label{Appendix A.12}
&\|\mathcal{P}_\varepsilon \psi_i^j\|_{L^\frac{2N}{N-2s}(\Omega_\varepsilon)}\leq C.
\end{eqnarray}
Moreover, we have
\begin{eqnarray}\label{Appendix A.13}
&\|\mathcal{P}_\varepsilon \psi_i^j\|_{L^\frac{2N}{N+2s}(\Omega_\varepsilon)}\leq C,
\end{eqnarray}
\begin{eqnarray}\label{Appendix A.14}
&\|\mathcal{P}_\varepsilon w_i\|_{L^\frac{2N}{N+2s}(\Omega_\varepsilon)}\leq
\left\{\aligned
&C\ \ \ \ \ \ \ \ \ \ \ \ \ \ \ \ \ \ \ \ \ \ \ \ \ \mbox{if}\ \ N>6s,\\
&C\varepsilon^{-\frac{(6s-N)\alpha_0}{2}}|\log\varepsilon|\ \ \ \ \ \mbox{if}\ \ N\leq 6s
\endaligned \right.
\end{eqnarray}
and
\begin{eqnarray}\label{Appendix A.15}
&\|\mathcal{P}_\varepsilon \psi_i^0\|_{L^\frac{2N}{N+2s}(\Omega_\varepsilon)}\leq
\left\{\aligned
&C\ \ \ \ \ \ \ \ \ \ \ \ \ \ \ \ \ \ \ \ \ \ \ \ \ \mbox{if}\ \ N>6s,\\
&C\varepsilon^{-\frac{(6s-N)\alpha_0}{2}}|\log\varepsilon|\ \ \ \ \ \mbox{if}\ \ N\leq 6s.
\endaligned \right.
\end{eqnarray}
\end{lemma}

\begin{lemma}\label{lemma A.4}
For $i=1,...,k$ and $j=1,...,N$, we have
\begin{eqnarray}\label{Appendix A.16}
&\|\mathcal{P}_\varepsilon \psi_i^j-\psi_i^j\|_{L^\frac{2N}{N-2s}(\Omega_\varepsilon)}\leq C\varepsilon^{\alpha_0\frac{N-2s+2}{2}}
\end{eqnarray}
and
\begin{eqnarray}\label{Appendix A.17}
&\|\mathcal{P}_\varepsilon \psi_i^0-\psi_i^0\|_{L^\frac{2N}{N-2s}(\Omega_\varepsilon)}\leq C\varepsilon^{\alpha_0\frac{N-2s}{2}}.
\end{eqnarray}
\end{lemma}

Similarly to Lemma A.2 in \cite{Bartsch1}, Lemma A.3 in \cite{Musso31} and Lemma C.5 in \cite{Chio12}, we obtain the following Lemma.

\begin{lemma}\label{(M17)-Lemma A}
For any $\eta>0$ and for any $\varepsilon_0>0$ there exists $C>0$ such that for any $(\mbox{\boldmath $\lambda$},\mbox{\boldmath $\sigma$})\in\mathcal{O}_\eta$ and $\varepsilon\in(0,\varepsilon_0)$ we have
\begin{eqnarray}\label{Appendix A.18}
&\left\|f_0\left(\sum\limits_{i=1}^ka_i\mathcal{P}_\varepsilon w_i\right)- \sum\limits_{i=1}^ka_if_0(w_i)\right\|_{L^\frac{2N}{N+2s}(\Omega_\varepsilon)}\leq
\left\{\aligned
&C\varepsilon^{\frac{N+2s}{2}\alpha_0}\ \ \ \ \ \ \ \ \ \ \ \ \ \ \ \ \ \ \ \ \mbox{if}\ \ N>6s,\\
&C(\varepsilon+\varepsilon^{(N-2s)\alpha_0}|\ln\varepsilon|)\ \ \ \ \ \mbox{if}\ \ N=6s,\\
&C\varepsilon^{(N-2s)\alpha_0}\ \ \ \ \ \ \ \ \ \ \ \ \ \ \ \ \ \ \mbox{if}\ \ N<6s,
\endaligned \right.
\end{eqnarray}
\begin{eqnarray}\label{Appendix A.19}
&\left\|f'_0\left(\sum\limits_{i=1}^k a_i\mathcal{P}_\varepsilon w_i\right)- \sum\limits_{i=1}^ka_if'_0(w_i)\right\|_{L^\frac{N}{2s}(\Omega_\varepsilon)}\leq C\varepsilon^{2s\alpha_0}.
\end{eqnarray}
\end{lemma}
\begin{proof}  We just prove the first inequality by using Lemma A.1 in \cite{Musso34}. The proof of (\ref{Appendix A.19}) is similar. Since $(\mbox{\boldmath $\lambda$},\mbox{\boldmath $\sigma$})\in\mathcal{O}_\eta$ it holds $|\sigma_i-\sigma_j|>\eta$ for any $i\neq j (i, j=1,...,k)$. We have
\begin{align*}
&\left\|f_0\left(\sum_{i=1}^ka_i\mathcal{P}_\varepsilon w_i\right)- \sum_{i=1}^ka_if_0(w_i)\right\|^{\frac{2N}{N+2s}}_{L^\frac{2N}{N+2s}(\Omega_\varepsilon)}\\
&=\int_{\Omega_\varepsilon}\left|\left|\sum_{i=1}^ka_i\mathcal{P}_\varepsilon w_i(y)\right|^{p-1}\left(\sum_{i=1}^ka_i\mathcal{P}_\varepsilon w_i(y)\right)-\sum_{i=1}^ka_iw_i^p(y)\right|^\frac{2N}{N+2s}dy\ \ \ ( \mbox{set}\ \ x=\varepsilon^{\alpha_0}y)\\
&=\int_\Omega\left|\left|\sum_{i=1}^ka_i\mathcal{P}_\varepsilon w_{\lambda_i\varepsilon^{\alpha_0},\sigma_i}(x)\right|^{p-1}\left(\sum_{i=1}^ka_i\mathcal{P}_\varepsilon w_{\lambda_i\varepsilon^{\alpha_0},\sigma_i}(x)\right)-\sum_{i=1}^ka_iw_{\lambda_i\varepsilon^{\alpha_0},\sigma_i}^p(x)\right|^\frac{2N}{N+2s}dx\\
&=\sum_{j=1}^k\int_{B(\sigma_j,\eta/2)}\left|\left|\sum_{i=1}^ka_i\mathcal{P}_\varepsilon w_{\lambda_i\varepsilon^{\alpha_0},\sigma_i}(x)\right|^{p-1}\left(\sum_{i=1}^ka_i\mathcal{P}_\varepsilon w_{\lambda_i\varepsilon^{\alpha_0},\sigma_i}(x)\right)-\sum_{i=1}^ka_iw_{\lambda_i\varepsilon^{\alpha_0},\sigma_i}^p(x)\right|^\frac{2N}{N+2s}dx\\
&+\int_{\Omega\setminus{\bigcup\limits_{j=1}^kB(\sigma_j,\eta/2)}}\left|\left|\sum_{i=1}^ka_i\mathcal{P}_\varepsilon w_{\lambda_i\varepsilon^{\alpha_0},\sigma_i}(x)\right|^{p-1}\left(\sum_{i=1}^ka_i\mathcal{P}_\varepsilon w_{\lambda_i\varepsilon^{\alpha_0},\sigma_i}(x)\right)-\sum_{i=1}^ka_iw_{\lambda_i\varepsilon^{\alpha_0},\sigma_i}^p(x)\right|^\frac{2N}{N+2s}dx.\\
\end{align*}
Firstly,
\begin{align*}
&\int_{\Omega\setminus{\bigcup\limits_{j=1}^kB(\sigma_j,\eta/2)}}\left|\left|\sum_{i=1}^ka_i\mathcal{P}_\varepsilon w_{\lambda_i\varepsilon^{\alpha_0},\sigma_i}(x)\right|^{p-1}\left(\sum_{i=1}^ka_i\mathcal{P}_\varepsilon w_{\lambda_i\varepsilon^{\alpha_0},\sigma_i}(x)\right)-\sum_{i=1}^ka_iw_{\lambda_i\varepsilon^{\alpha_0},\sigma_i}^p(x)\right|^\frac{2N}{N+2s}dx\\
&\leq C\sum_{i=1}^k\int_{\Omega\setminus{\bigcup\limits_{j=1}^kB(\sigma_j,\eta/2)}}w_{\lambda_i\varepsilon^{\alpha_0},\sigma_i}^{p\frac{2N}{N+2s}}dx\\
&\leq C\sum_{i=1}^k(\lambda_i\varepsilon^{\alpha_0})^N\\
&\leq C\varepsilon^{N{\alpha_0}}.
\end{align*}
Secondly,
\begin{align*}
&\int_{B(\sigma_j,\eta/2)}\left|\left|\sum_{i=1}^ka_i\mathcal{P}_\varepsilon w_{\lambda_i\varepsilon^{\alpha_0},\sigma_i}(x)\right|^{p-1}\left(\sum_{i=1}^ka_i\mathcal{P}_\varepsilon w_{\lambda_i\varepsilon^{\alpha_0},\sigma_i}(x)\right)-\sum_{i=1}^ka_iw_{\lambda_i\varepsilon^{\alpha_0},\sigma_i}^p(x)\right|^\frac{2N}{N+2s}dx\\
&\leq C\int_{B(\sigma_j,\eta/2)}\left|\left(\sum_{i=1}^k\mathcal{P}_\varepsilon w_{\lambda_i\varepsilon^{\alpha_0},\sigma_i}(x)\right)^p-w_{\lambda_j\varepsilon^{\alpha_0},\sigma_j}^p(x)\right|^\frac{2N}{N+2s}dx\\
&\ \ \ +C\sum\limits_{i=1, i\neq j}^k\int_{B(\sigma_j,\eta/2)}|w_{\lambda_i\varepsilon^{\alpha_0},\sigma_i}^p(x)|^\frac{2N}{N+2s}dx\\
&\leq C\int_{B(\sigma_j,\eta/2)}\left|\mathcal{P}_\varepsilon w_{\lambda_j\varepsilon^{\alpha_0},\sigma_j}^p(x) -w_{\lambda_j\varepsilon^{\alpha_0},\sigma_j}^p(x)\right|^\frac{2N}{N+2s}dx\\
&\ \ \ +C\sum\limits_{i=1, i\neq j}^k\int_{B(\sigma_j,\eta/2)}|w_{\lambda_i\varepsilon^{\alpha_0},\sigma_i}^p(x)|^\frac{2N}{N+2s}dx+C\varepsilon^{N{\alpha_0}}.
\end{align*}
It is easy to see that
\begin{align*}
&\sum\limits_{i=1, i\neq j}^k\int_{B(\sigma_j,\eta/2)}|w_{\lambda_i\varepsilon^{\alpha_0},\sigma_i}^p(x)|^\frac{2N}{N+2s}dx\\
&\leq \sum\limits_{i=1, i\neq j}^k\int_{B(\sigma_j,\eta/2)}\left|\frac{\lambda_i\varepsilon^{\alpha_0}}{(\lambda_i\varepsilon^{\alpha_0})^2+|x-\sigma_i|^2}\right|^Ndx\\
&\leq C\varepsilon^{N{\alpha_0}}.
\end{align*}
For $N>6s$, by using Lemma \ref{lemma A.1} and the mean value theorem, we get that
\begin{align*}
&\int_{B(\sigma_j,\eta/2)}\left|\mathcal{P}_\varepsilon w_{\lambda_j\varepsilon^{\alpha_0},\sigma_j}^p(x) -w_{\lambda_j\varepsilon^{{\alpha_0},\sigma_j}}^p(x)\right|^\frac{2N}{N+2s}dx\\
&=p\int_{B(\sigma_j,\eta/2)}|(w_{\lambda_j\varepsilon^{\alpha_0},\sigma_j}+\theta(x)(\mathcal{P}_\varepsilon w_{\lambda_j\varepsilon^{\alpha_0},\sigma_j}(x) -w_{\lambda_j\varepsilon^{\alpha_0},\sigma_j}(x)))^{p-1}\\
&\ \ \ \ \times(\mathcal{P}_\varepsilon w_{\lambda_j\varepsilon^{\alpha_0},\sigma_j}(x)-w_{\lambda_j\varepsilon^{\alpha_0},\sigma_j}(x))|^\frac{2N}{N+2s}dx\\
&\leq C\varepsilon^{N{\alpha_0}}.
\end{align*}
Therefore if $N>6s$, we have
\begin{align*}
&\left\|f_0\left(\sum_{i=1}^ka_i\mathcal{P}_\varepsilon w_i\right)- \sum_{i=1}^ka_if_0(w_i)\right\|_{L^\frac{2N}{N+2s}(\Omega_\varepsilon)}\leq C\varepsilon^{{\frac{N+2s}{2}\alpha_0}}.
\end{align*}

Moreover, if $N=6s$, we have
\begin{align*}
&\int_{B(\sigma_j,\eta/2)}\left|\mathcal{P}_\varepsilon w_{\lambda_j\varepsilon^{\alpha_0},\sigma_j}^p(x) -w_{\lambda_j\varepsilon^{{\alpha_0},\sigma_j}}^p(x)\right|^\frac{2N}{N+2s}dx\\
&=p\int_{B(\sigma_j,\eta/2)}|(w_{\lambda_j\varepsilon^{\alpha_0},\sigma_j}+\theta(x)(\mathcal{P}_\varepsilon w_{\lambda_j\varepsilon^{\alpha_0},\sigma_j}(x) -w_{\lambda_j\varepsilon^{\alpha_0},\sigma_j}(x)))^{p-1}\\
&\ \ \ \ \times(\mathcal{P}_\varepsilon w_{\lambda_j\varepsilon^{\alpha_0},\sigma_j}(x) -w_{\lambda_j\varepsilon^{\alpha_0},\sigma_j}(x))|^\frac{2N}{N+2s}dx\\
&\leq C\int_{B(\sigma_j,\eta/2)}\left|\frac{\lambda_j\varepsilon^{\alpha_0}}{({\lambda_j\varepsilon^{\alpha_0}})^2+|x-\sigma_j|^2}\right|^\frac{4sN}{N+2s}
(\varepsilon^{\alpha_0})^\frac{N(N-2s)}{N+2s}dx\\
&=C\int_0^\frac{1}{\varepsilon^{\alpha_0}}\left(\frac{1}{1+\rho^2}\right)^\frac{4sN}{N+2s}(\varepsilon^{\alpha_0})^{\frac{N(N-2s)}{N+2s}+N-\frac{4sN}{N+2s}}\rho^{N-1}
d\rho\\
&\leq C\varepsilon^{\frac{2N}{N+2s}}|\ln\varepsilon|.
\end{align*}
On the other hand, if $2s<N<6s$, using the substitution $x-\sigma_j=\lambda_j\varepsilon^{\alpha_0}z$, we get
\begin{align*}
&\int_{B(\sigma_j,\eta/2)}\left|\mathcal{P}_\varepsilon w_{\lambda_j\varepsilon^{\alpha_0},\sigma_j}^p(x) -w_{\lambda_j\varepsilon^{{\alpha_0},\sigma_j}}^p(x)\right|^\frac{2N}{N+2s}dx\\
&\leq C\varepsilon^{\frac{2N}{N+2s}}\int_{\mathbb R^N}\frac{1}{(1+z^2)^\frac{4Ns}{N+2s}}dz\\
&\leq C\varepsilon^{\frac{2N}{N+2s}}.
\end{align*}
\end{proof}
\begin{lemma}\label{lemma A.6}
For any $\eta>0$ and $\varepsilon_0>0$, there exists $C>0$ such that for any $(\mbox{\boldmath $\lambda$},\mbox{\boldmath $\sigma$})\in\mathcal{O}_\eta$ and for any $\varepsilon\in(0,\varepsilon_0)$ we have for $h=1,...,k$ and $j=0,1,...,N$
\begin{eqnarray}\label{Appendix A.20}
&\left\|\left[f'_0\left(\sum\limits_{i=1}^ka_i\mathcal{P}_\varepsilon w_i\right)- \sum\limits_{i=1}^ka_if'_0(w_i)\right]\mathcal{P}_\varepsilon \psi_h^j\right\|_{L^\frac{2N}{N+2s}(\Omega_\varepsilon)}\leq
C\varepsilon^{\frac{N+2s}{2}\alpha_0}.
\end{eqnarray}
\end{lemma}

\begin{proof}
 Since $(\mbox{\boldmath $\lambda$},\mbox{\boldmath $\sigma$})\in\mathcal{O}_\eta$ it holds $|\sigma_i-\sigma_j|>\eta$ for any $i\neq j(i, j=1,...,k)$, by using Lemma \ref{lemma A.4} and Lemma \ref{(M17)-Lemma A}, we have
\begin{align*}
&\int_{\Omega_\varepsilon}\left(\left|f'_0\left(\sum_{i=1}^ka_i\mathcal{P}_\varepsilon w_i\right)- \sum_{i=1}^ka_if'_0(w_i)\right||\mathcal{P}_\varepsilon \psi_h^j|\right)^\frac{2N}{N+2s}dx\\
&\leq\int_{\Omega_\varepsilon}\left(\left|f'_0\left(\sum_{i=1}^ka_i\mathcal{P}_\varepsilon w_i\right)-\sum_{i=1}^ka_if'_0(w_i)\right||\mathcal{P}_\varepsilon \psi_h^j-\psi_h^j|\right)^\frac{2N}{N+2s}dx\\
&\ \ \ +\int_{\Omega_\varepsilon}\left(\left|f'_0\left(\sum_{i=1}^ka_i\mathcal{P}_\varepsilon w_i\right)- \sum_{i=1}^ka_if'_0(w_i)\right||\psi_h^j|\right)^\frac{2N}{N+2s}dx\\
&\leq\left\|f'_0\left(\sum_{i=1}^ka_i\mathcal{P}_\varepsilon w_i\right)- \sum_{i=1}^ka_if'_0(w_i)\right\|_{L^\frac{N}{2s}(\Omega_\varepsilon)}^\frac{2N}{N+2s}\|\mathcal{P}_\varepsilon \psi_h^j-\psi_h^j\|_{L^\frac{2N}{N-2s}(\Omega_\varepsilon)}^\frac{2N}{N+2s}\\
&\ \ \ +\int_{\Omega_\varepsilon}\left(\left|f'_0\left(\sum_{i=1}^ka_i\mathcal{P}_\varepsilon w_i\right)- \sum_{i=1}^ka_if'_0(w_i)\right||\psi_h^j|\right)^\frac{2N}{N+2s}dx.
\end{align*}

Now by using (\ref{Appendix A.3}) and (\ref{Appendix A.4}) we have
\begin{align*}
&\int_{\Omega_\varepsilon}\left(\left|f'_0\left(\sum_{i=1}^ka_i\mathcal{P}_\varepsilon w_i\right)- \sum_{i=1}^ka_if'_0(w_i)\right||\psi_h^j|\right)^\frac{2N}{N+2s}dy\ \ \ ( \mbox{set}\ \ x=\varepsilon^{\alpha_0}y)\\
&=\varepsilon^\frac{2N\alpha_0}{N+2s}\int_\Omega\left(\left|f'_0\left(\sum_{i=1}^ka_i\mathcal{P}_\varepsilon w_{\lambda_i\varepsilon^{\alpha_0},\sigma_i}\right)- \sum_{i=1}^ka_if'_0(w_{\lambda_i\varepsilon^{\alpha_0},\sigma_i})\right||\psi_{\lambda_h\varepsilon^{\alpha_0},\sigma_h}^j|\right)^\frac{2N}{N+2s}dx\\
&\leq\varepsilon^\frac{2N\alpha_0}{N+2s}\int_{B(\sigma_h,\eta/2)}\left(\left|f'_0\left(\sum_{i=1}^ka_i\mathcal{P}_\varepsilon w_{\lambda_i\varepsilon^{\alpha_0},\sigma_i}\right)- \sum_{i=1}^ka_if'_0(w_{\lambda_i\varepsilon^{\alpha_0},\sigma_i})\right||\psi_{\lambda_h\varepsilon^{\alpha_0},\sigma_h}^j|\right)^\frac{2N}{N+2s}dx\\
&\ \ \ +\varepsilon^\frac{2N\alpha_0}{N+2s}\int_{\Omega\setminus B(\sigma_h,\eta/2)}\left(\left|f'_0\left(\sum_{i=1}^ka_i\mathcal{P}_\varepsilon w_{\lambda_i\varepsilon^{\alpha_0},\sigma_i}\right)- \sum_{i=1}^ka_if'_0(w_{\lambda_i\varepsilon^{\alpha_0},\sigma_i})\right||\psi_{\lambda_h\varepsilon^{\alpha_0},\sigma_h}^j|\right)^\frac{2N}{N+2s}dx.\\
\end{align*}
Firstly, by Lemma \ref{lemma A.1}, we get
\begin{align*}
&\int_{B(\sigma_h,\eta/2)}\left(\left|f'_0\left(\sum_{i=1}^ka_i\mathcal{P}_\varepsilon w_{\lambda_i\varepsilon^{\alpha_0},\sigma_i}\right)- \sum_{i=1}^ka_if'_0(w_{\lambda_i\varepsilon^{\alpha_0},\sigma_i})\right||\psi_{\lambda_h\varepsilon^{\alpha_0},\sigma_h}^j|\right)^\frac{2N}{N+2s}dx\\
&\leq C\int_{B(\sigma_h,\eta/2)}\left(\left|f'_0\left(\sum_{i=1}^ka_i\mathcal{P}_\varepsilon w_{\lambda_i\varepsilon^{\alpha_0},\sigma_i}\right)- a_hf'_0(w_{\lambda_h\varepsilon^{\alpha_0},\sigma_h})\right||\psi_{\lambda_h\varepsilon^{\alpha_0},\sigma_h}^j|\right)^\frac{2N}{N+2s}dx\\
&\ \ \ +C\sum\limits_{i=1, i\neq h}^{k}\int_{B(\sigma_h,\eta/2)}|f'_0(w_{\lambda_i\varepsilon^{\alpha_0},\sigma_i})||\psi_{\lambda_h\varepsilon^{\alpha_0},\sigma_h}^j|^\frac{2N}{N+2s}dx\\
&\leq C\int_{B(\sigma_h,\eta/2)}|\mathcal{P}_\varepsilon w_{\lambda_h\varepsilon^{\alpha_0},\sigma_h}- w_{\lambda_h\varepsilon^{\alpha_0},\sigma_h}|^\frac{8Ns}{(N+2s)(N-2s)}|\psi_{\lambda_h\varepsilon^{\alpha_0},\sigma_h}^j|^\frac{2N}{N+2s}dx\\
&\ \ \ +C\sum\limits_{i=1, i\neq h}^{k}\int_{B(\sigma_h,\eta/2)}|w_{\lambda_i\varepsilon^{\alpha_0},\sigma_i}|^\frac{8Ns}{(N+2s)(N-2s)}|\psi_{\lambda_h\varepsilon^{\alpha_0},\sigma_h}^j|^\frac{2N}{N+2s}dx\\
&\leq C\varepsilon^\frac{\alpha_0N^2}{N+2s}.
\end{align*}
Secondly, we have
\begin{align*}
&\int_{\Omega\setminus B(\sigma_h,\eta/2)}\left(\left|f'_0\left(\sum_{i=1}^ka_i\mathcal{P}_\varepsilon w_{\lambda_i\varepsilon^{\alpha_0},\sigma_i}\right)- \sum_{i=1}^ka_if'_0(w_{\lambda_i\varepsilon^{\alpha_0},\sigma_i})\right||\psi_{\lambda_h\varepsilon^{\alpha_0},\sigma_h}^j|\right)^\frac{2N}{N+2s}dx\\
&\leq\sum_{i=1}^k\int_{\Omega\setminus B(\sigma_h,\eta/2)}|w_{\lambda_i\varepsilon^{\alpha_0},\sigma_i}|^\frac{8Ns}{(N+2s)(N-2s)}|\psi_{\lambda_h\varepsilon^{\alpha_0},\sigma_h}^j|^\frac{2N}{N+2s}dx\\
&\leq C\varepsilon^\frac{\alpha_0N^2}{N+2s}.
\end{align*}
By the above estimations the proof is complete.
\end{proof}

\begin{lemma}\label{lemma A.7}
For any $\eta>0$ and $\varepsilon_0>0$ there exists $C>0$ such that for any $(\mbox{\boldmath $\lambda$},\mbox{\boldmath $\sigma$})\in\mathcal{O}_\eta$ and for any $\varepsilon\in(0,\varepsilon_0)$ we have
\begin{eqnarray}\label{lemma A.7-1}
&\left\| f_\varepsilon\left(\sum\limits_{i=1}^ka_i\mathcal{P}_\varepsilon w_i\right)-f_0\left(\sum\limits_{i=1}^ka_i\mathcal{P}_\varepsilon w_i\right)\right\|_{L^\frac{2N}{N+2s}(\Omega_\varepsilon)}\leq C\varepsilon|\ln\varepsilon|.
\end{eqnarray}
\begin{eqnarray}\label{lemma A.7-2}
&\left\|f'_\varepsilon\left(\sum\limits_{i=1}^ka_i\mathcal{P}_\varepsilon w_i\right)-f'_0\left(\sum\limits_{i=1}^ka_i\mathcal{P}_\varepsilon w_i\right)\right\|_{L^\frac{N}{2s}(\Omega_\varepsilon)}\leq C\varepsilon|\ln\varepsilon|,
\end{eqnarray}
\end{lemma}
\begin{proof}  Let us prove (\ref{lemma A.7-1}). The proof of (\ref{lemma A.7-2}) is similar.  Since $(\mbox{\boldmath $\lambda$},\mbox{\boldmath $\sigma$})\in\mathcal{O}_\eta$ it holds $|\sigma_i-\sigma_j|>\eta$ for any $i\neq j(i, j=1,...,k)$. We have
\begin{align*}
&\left\|f_\varepsilon\left(\sum\limits_{i=1}^ka_i\mathcal{P}_\varepsilon w_i\right)-f_0\left(\sum\limits_{i=1}^ka_i\mathcal{P}_\varepsilon w_i\right)\right\|_{L^{\frac{2N}{N+2s}}(\Omega_\varepsilon)}^{\frac{2N}{N+2s}}\\
&=\int_{\Omega_\varepsilon}\left|\left|\sum\limits_{i=1}^ka_i\mathcal{P}_\varepsilon w_i\right|^{p-\varepsilon}-\left|\sum\limits_{i=1}^ka_i\mathcal{P}_\varepsilon w_i\right|^p\right|^{\frac{2N}{N+2s}}dx\\
&=\int_{\Omega_\varepsilon}\left|\varepsilon\left|\sum\limits_{i=1}^ka_i\mathcal{P}_\varepsilon w_i\right|^{p-t\varepsilon}\ln\left|\sum\limits_{i=1}^ka_i\mathcal{P}_\varepsilon w_i\right|\right|^{\frac{2N}{N+2s}}dx\ \ \ ( \mbox{set}\ \ y=\varepsilon^{\alpha_0}x)\\
&=\varepsilon^{\frac{2N}{N+2s}(\frac{p-t\varepsilon}{2}+1)-\frac{N}{N-2s}}\int_{\Omega}\left|\left|\sum\limits_{i=1}^ka_i\mathcal{P}_\varepsilon w_{\lambda_i\varepsilon^{\alpha_0},\sigma_i}\right|^{p-t\varepsilon}\ln\left|\varepsilon^{\frac{(N-2s)\alpha_0}{2}}\sum\limits_{i=1}^ka_i\mathcal{P}_\varepsilon w_{\lambda_i\varepsilon^{\alpha_0},\sigma_i}\right|\right|^{\frac{2N}{N+2s}}dy\\
&=\varepsilon^{\frac{2N}{N+2s}(\frac{p-t\varepsilon}{2}+1)-\frac{N}{N-2s}}\Bigg[\sum_{j=1}^k\int_{B(\sigma_j,\eta/2)}\left|\left|\sum\limits_{i=1}^ka_i\mathcal{P}_\varepsilon w_{\lambda_i\varepsilon^{\alpha_0},\sigma_i}\right|^{p-t\varepsilon}\ln\left|\varepsilon^{\frac{(N-2s)\alpha_0}{2}}\sum\limits_{i=1}^ka_i\mathcal{P}_\varepsilon w_{\lambda_i\varepsilon^{\alpha_0},\sigma_i}\right|\right|^{\frac{2N}{N+2s}}dy\\
&\ \ \ \ +\int_{\Omega\setminus{\bigcup\limits_{j=1}^kB(\sigma_j,\eta/2)}}\left|\left|\sum\limits_{i=1}^ka_i\mathcal{P}_\varepsilon w_{\lambda_i\varepsilon^{\alpha_0},\sigma_i}\right|^{p-t\varepsilon}\ln\left|\varepsilon^{\frac{(N-2s)\alpha_0}{2}}\sum\limits_{i=1}^ka_i\mathcal{P}_\varepsilon w_{\lambda_i\varepsilon^{\alpha_0},\sigma_i}\right|\right|^{\frac{2N}{N+2s}}dy\Bigg].\\
\end{align*}

Firstly, we have
\begin{align*}
&\int_{\Omega\setminus{\bigcup\limits_{j=1}^kB(\sigma_j,\eta/2)}}\varepsilon^{\frac{2N}{N+2s}(\frac{p-t\varepsilon}{2}+1)-\frac{N}{N-2s}}\left|\left|\sum\limits_{i=1}^ka_i\mathcal{P}_\varepsilon w_{\lambda_i\varepsilon^{\alpha_0},\sigma_i}\right|^{p-t\varepsilon}\ln\left|\varepsilon^{\frac{(N-2s)\alpha_0}{2}}\sum\limits_{i=1}^ka_i\mathcal{P}_\varepsilon w_{\lambda_i\varepsilon^{\alpha_0},\sigma_i}\right|\right|^{\frac{2N}{N+2s}}dy\\
&\leq C\varepsilon^{\frac{2N}{N+2s}(\frac{p-t\varepsilon}{2}+1)-\frac{N}{N-2s}}\sum\limits_{i=1}^k(\lambda_i\varepsilon^{\alpha_0})^{\frac{N-2s}{2}
(p-t\varepsilon)\frac{2N}{N+2s}}\left|\ln\left(\varepsilon^{\frac{(N-2s)\alpha_0}{2}}\sum
\limits_{i=1}^k(\lambda_i\varepsilon^{\alpha_0})^{\frac{N-2s}{2}}\right)\right|^{\frac{2N}{N+2s}}\\
&\leq C |\varepsilon\ln\varepsilon|^{\frac{2N}{N+2s}}
\end{align*}
Moreover, using the substitution $x-\sigma_i=\lambda_i\varepsilon^{\alpha_0}z$, we get
\begin{align*}
&\int_{B(\sigma_j,\eta/2)}\varepsilon^{\frac{2N}{N+2s}(\frac{p-t\varepsilon}{2}+1)-\frac{N}{N-2s}}\left|\left|\sum\limits_{i=1}^ka_i\mathcal{P}_\varepsilon w_{\lambda_i\varepsilon^{\alpha_0},\sigma_i}\right|^{p-t\varepsilon}\ln\left|\varepsilon^{\frac{(N-2s)\alpha_0}{2}}\sum\limits_{i=1}^ka_i\mathcal{P}_\varepsilon w_{\lambda_i\varepsilon^{\alpha_0},\sigma_i}\right|\right|^{\frac{2N}{N+2s}}dy\\
&\leq C\sum\limits_{i=1}^k\int_{B(\sigma_j,\eta/2)}\varepsilon^{\frac{2N}{N+2s}(\frac{p-t\varepsilon}{2}+1)-\frac{N}{N-2s}}
\left|w_{\lambda_i\varepsilon^{\alpha_0},\sigma_i}\right|^{\frac{2N(p-t\varepsilon)}{N+2s}}
\left|\ln\left|\varepsilon^{\frac{(N-2s)\alpha_0}{2}}\sum\limits_{i=1}^k\mathcal{P}_\varepsilon w_{\lambda_i\varepsilon^{\alpha_0},\sigma_i}\right|\right|^{\frac{2N}{N+2s}}dy\\
&\leq C\sum\limits_{i=1}^k\int_{\mathbb R^N}\varepsilon^{\frac{2N}{N+2s}(\frac{p-t\varepsilon}{2}+1)-\frac{N}{N-2s}}(\lambda_i\varepsilon^{\alpha_0})^{N-\frac{N(N-2s)}{N+2s}(p-t\varepsilon)}
\left(\frac{1}{1+z^2}\right)^{\frac{N(N-2s)}{N+2s}(p-t\varepsilon)}\\
&\ \ \ \times\left|\ln\left|\lambda_i^{-\frac{N-2s}{2}}
\Bigg[\left(\frac{1}{1+z^2}\right)^\frac{N-2s}{2}+\sum\limits_{j=1, i\neq j}^{k}\left(\frac{\lambda_i\varepsilon^{\alpha_0}\lambda_j\varepsilon^{\alpha_0}}{(\lambda_j\varepsilon^{\alpha_0})^2+|\lambda_i\varepsilon^{\alpha_0}z+\sigma_i-
\sigma_j|^2}\right)^\frac{N-2s}{2}\Bigg]\right|\right|^{\frac{2N}{N+2s}}dz\\
&\leq C |\varepsilon\ln\varepsilon|^{\frac{2N}{N+2s}}.
\end{align*}
\end{proof}



\begin{thebibliography}{a}
\bibitem{Bartsch1} T. Bartsch, A. M. Micheletti, A Pistoia, On the existence and the profile of nodal solutions of elliptic equations involving critical growth, Calc. Var. Partial Differential Equations 3 (2006) 265-282.
\bibitem{Caffarelli2} L. Caffarelli, L. Silvestre, An extension problem related to the fractional Laplacian, Comm. Partial Differential Equations 32 (2007) 1245-1260.
\bibitem{Tan17} X. Cabr\'{e}, J. Tan, Positive solutions of nonlinear problems involving the square root of the Laplacian,
Adv. Math. 224 (2010) 2052-2093.
\bibitem{Jing13} J. Tan, The Brezis-Nirenberg type problem involving the square root of the
Laplacian, Calc. Var. Partial Differential Equations, 42 (2011) 21-41.
\bibitem{Chio12} W. Chio, S. Kim, K.-A. Lee, Asymptotic behavior of solutions for nonlinear elliptic problem with the fractional Laplacian, J. Funct. Anal. 11 (2014) 6531-6598.
\bibitem{Rois15} L.F.T. R\'{i}os, Two problems in nonlinear PDEs: existence in supercritical elliptic equations and symmetry for a hypo-elliptic operator [D], Universided de 511-540. Chile, Chile, 2014.
\bibitem{Rey3} O. Rey, The role of the Green's function in a nonlinear elliptic equation involving the critical Sobolev exponent, J. Funct. Anal. 89 (1990) 1-52.
\bibitem{Han4} Z.C. Han, Asymptotic approach to singular solutions for nonlinear elliptic equations involving critical
Sobolev exponent, Ann. Inst. H. Poincar\'{e} Anal. Non Lin\'{e}aire 8 (1991) 159-174.
\bibitem{Brezis5} H. Br\'{e}zis, L. A. Peletier, Asymptotic for elliptic equations involving critical growth, Partial Differential Equations and the calculus of variations ,vol. I, Progr. Nonlinear Diff. Equ. Appl. Birh\"{a}user, Boston, MA 1 (1989) 149-192.
\bibitem{Rey6} O. Rey, Blow-up points of solutions to elliptic equations with limiting nonlinearity, Differential Integral
Equations 4 (1991) 1155-1167.
\bibitem{Bahri7} A. Bahri, Y. Li, O. Rey, On a variational problem with lack of compactness: the topological effect of the
critical points at infinity, Calc. Var. Partial Differential Equations 3 (1995) 67-93.
\bibitem{Grossi8} M. Grossi, F. Takahashi, Nonexistence of multi-bubble solutions to some elliptic equations on convex
domains, J. Functional Analysis 259 (2010) 904-917.
\bibitem{Ben9} M. Ben Ayed , K. El Mehdi, F. Pacella, Classification of low energy sign-changing solutions of an almost critical equations, J. Funct. Anal. 50 (2007) 343-373.
\bibitem{Bartsch10} T. Bartsch, T. D'Aprile, A. Pistoia, Multi-bubble nodal solutions for slightly subcritical elliptic problems in domains with symmetries, Ann. Inst. H. Poincar\'{e} Anal. Non Lin\'{e}aire 30 (2013) 1027-1047.
\bibitem{Bartsch11} T. Bartsch, T. D'Aprile, A. Pistoia, On the profile of sign changing solutions of an almost critical problem in the ball, Bull. London Math. Soc. 45 (2013) 1246-1258.
\bibitem{Colorado16} C. Br\'{a}ndle, E. Colorado, A. de Pablo, U. S¨¢nchez, A concave-convex elliptic problem involving the
    fractional Laplacian, Proc. Roy. Soc. Edinburgh Sect. A 143 (2013) 39-71.
\bibitem{Capella18} A. Capella, J. D\'{a}vila, L. Dupaigne, Y. Sire, Regularity of radial extremal solutions for some non-local
semilinear equations, Comm. Partial Differential Equations 36 (2011) 1353-1384.
\bibitem{Kim19} S. Kim, K.-A. Lee, H\"{o}lder estimates for singular non-local parabolic equations, J. Funct. Anal. 261
(2011) 3482-3518.
\bibitem{Tan20} J. Tan, Positive solutions for non local elliptic problems, Discrete Contin. Dyn. Syst. 33 (2013)
837-859.
\bibitem{Stinga21} P.S. Stinga, J.L. Torrea, Extension problem and Harnack's inequality for some fractional operators,
Comm. Partial Differential Equations 35 (2010) 2092-2122.
\bibitem{Barrios22}B. Barrios, E. Colorado, A. de Palbo, U. S\'{a}nchez, On some critical problems for the fractional Laplacian operator, J. Differential Equations 252 (2012) 6133-6162.
\bibitem{Cotsiolis23} A. Cotsiolis, N.K. Tavoularis, Best constants for Sobolev inequalities for higher order fractional
derivatives, J. Math. Anal. Appl. 295 (2004) 225-236.
\bibitem{Chen24} W. Chen, C. Li, B. Ou, Classification of solutions for an integral equation, Comm. Pure Appl. Math.
59 (2006) 330-343.
\bibitem{Li25} Y.Y. Li, Remark on some conformally invariant integral equations: the method of moving spheres,
J. Eur. Math. Soc. 6 (2004) 153-180.
\bibitem{Li26} Y.Y. Li, M. Zhu, Uniqueness theorems through the method of moving spheres, Duke Math. J. 80
(1995) 383-417.
\bibitem{Carlen27} E.A. Carlen, M. Loss, Extremals of functionals with competing symmetries, J. Funct. Anal. 88
(1990) 437-456.
\bibitem{Frank28} R.L. Frank, E.H. Lieb, Inversion positivity and the sharp Hardy-Littlewood-Sobolev inequality,
Calc. Var. Partial Differential Equations 39 (2010) 85-99.
\bibitem{Lieb29} E.H. Lieb, Sharp constants in the Hardy-Littlewood-Sobolev and related inequalities, Ann. of Math.
118 (1983) 349-374.
\bibitem{Xiao30} J. Xiao, A sharp Sobolev trace inequality for the fractional-order derivatives, Bull. Sci. Math. 130
(2006) 87-96.
\bibitem{Musso31} M. Musso, A. Pistoia, Multispike solutions for a nonlinear elliptic problem involving the critical
Sobolev exponent, Indiana Univ. Math. J. 51 (2002) 541-579.
\bibitem{del32} J. D\'{a}vila, M. del Pino, Y. Sire, Non degeneracy of the bubble in the critical case for nonlocal equations, Proc. Amer. Math. Soc. 141 (2013) 3865-3870.
\bibitem{Musso34} M. Musso, A. Pistoia, Tower of bubbles for almost critical problems in general domains, J. Math. Pures Appl. 93 (2010) 1-40.
\bibitem{Del33} M. Del Pino, P. Felmer, M. Musso, Two-bubble solutions in the super-critical Bahri-Coron's problem, Calc. Var. Partial Differential Equations 16 (2003) 113-145.

\end{thebibliography}
\end{document}